\documentclass{article}
\usepackage[utf8]{inputenc}
\usepackage{amsmath,amssymb,amsthm}

\usepackage{tikz,caption,subcaption}
\usetikzlibrary{decorations.pathreplacing,calligraphy}

\theoremstyle{plain}
\newtheorem{theorem}{Theorem}[section]
\newtheorem{lemma}[theorem]{Lemma}
\newtheorem{claim}[theorem]{Claim}

\newtheorem{conjecture}[theorem]{Conjecture}

\date{}
\title{\vspace{-0.9cm}Maximal Chordal Subgraphs}
\author{
Lior Gishboliner\thanks{Department of Mathematics, ETH, Z\"urich, Switzerland. Research supported in part by SNSF grant 200021\_196965. Email: \textbf{\{lior.gishboliner, benjamin.sudakov\}@math.ethz.ch}.}
\and Benny Sudakov\footnotemark[1]
}

\begin{document}

\maketitle

\begin{abstract}
    A chordal graph is a graph with no induced cycles of length at least $4$. 
    Let $f(n,m)$ be the maximal integer such that every graph with $n$ vertices and $m$ edges has a chordal subgraph with at least $f(n,m)$ edges. In 1985 Erd\H{o}s and Laskar posed the problem of estimating $f(n,m)$. In the late '80s, Erd\H{o}s, Gy\'arf\'as, Ordman and Zalcstein determined the value of $f(n,n^2/4+1)$ and made a conjecture on the value of $f(n,n^2/3+1)$. In this paper we prove this conjecture and answer the question of Erd\H{o}s and Laskar, determining $f(n,m)$ asymptotically for all $m$ and exactly for $m \leq n^2/3+1$. 
\end{abstract}

\section{Introduction}
One of the central questions in extremal combinatorics can be formulated as follows. Given a graph $G$ and a property $\cal P$, what is the maximal subgraph of $G$ one can find which satisfies this property. The study of this problem goes back to the work of Tur\'an in 1941, whose theorem states that the largest subgraph of the $n$-vertex complete graph with no clique of size $k+1$ is the complete $k$-partite graph with sides as equal as possible. This graph is called the Tur\'an graph. We denote it by $T_k(n)$ and its size by $t_k(n)$.
Tur\'an's theorem is the starting point of extremal graph theory and has inspired extensive research. One such research direction studies which other (more elaborate) structures must appear in a graph with more than $t_k(n)$ edges. 
For example, a series of works determined how many $(k+1)$-cliques must exist in a graph with $t_k(n) + a$ edges (for a suitable range of $a$) \cite{Erdos62,Erdos69,LS}. Other examples are results on finding many $(k+1)$-cliques which share one or more vertices \cite{Erdos62,Edwards_n/6,KN79,EFGG}, and results on finding $(k+1)$-cliques with large degree sum \cite{Edwards76,Edwards77,Faudree92,BN}. 

In this paper we study the Tur\'an type problem for chordal graphs. A graph is called chordal if it contains no induced cycle of length at least $4$. Chordal graphs are one of the most studied classes in graph theory and have numerous applications, for example in semidefinite optimization (see the survey \cite{semidefinite}) and evolutionary trees (see \cite{EK}). In 1985, Erd\H{o}s and Laskar \cite{EL} 
asked to determine the maximum integer $f(n,m)$ such that every graph with $n$ vertices and $m$ edges contains a chordal subgraph with at least $f(n,m)$ edges. To put this question under the umbrella of classical extremal graph theory, one needs to consider equivalent definitions of chordal graphs. It is well-known that a graph is chordal if and only if it can be constructed from a single-vertex graph by repeatedly adding a vertex and connecting it to a clique of the current graph\footnote{A related fact is that a graph is chordal if and only if it has a tree-decomposition in which the bags are cliques. So chordal graphs can be thought of as "trees of cliques".} (this is called a perfect elimination ordering), see \cite[Chapter 4]{Golumbic}. So if $G$ is a triangle-free graph, then every chordal subgraph of $G$ must be a forest. More generally, if $G$ has no cliques of size $k+1$, then every chordal subgraph of $G$ has at most $(k-1)(n-k+1) + \binom{k-1}{2} =  (k-1)n - \binom{k}{2}$ edges. In particular, this bound applies to $k$-partite graph. Another way of proving this bound for $k$-partite graphs is to observe that if $G$ is $k$-partite with parts $V_1,\dots,V_k$ and $H$ is a chordal subgraph of $G$, then $e_H(V_i,V_j) \leq |V_i| + |V_j| - 1$ for every $i < j$ (because a chordal subgraph of a bipartite graph must be a forest). Hence, $e(H) \leq \sum_{i < j}(|V_i| + |V_j| - 1) = (k-1)n - \binom{k}{2}$. % This argument will be useful later on. 

The above discussion shows that if $m \leq t_k(n)$ then $f(n,m) \leq (k-1)n  - \binom{k}{2}$. It is natural to guess that the value of $f(n,m)$ ``jumps'' as $m$ increases from $t_k(n)$ to $m = t_k(n) + 1$, because at this point the graph must contain $(k+1)$-cliques. Erd\H{o}s and Laskar \cite{EL} proved that this is indeed the case for $k=2$, showing that $f(n,t_2(n)+1) \geq (1+\varepsilon)n$. In the late 80's, Erd\H{o}s, Gy\'arf\'as, Ordman and Zalcstein \cite{EGOZ} determined the value of $f(n,t_2(n)+1)$ exactly for even $n$, showing that $f(n,\frac{n^2}{4}+1) = \frac{3n}{2} - 1$. This bound is achieved by the graph $T_2(n) + e$, obtained by adding an edge to the Tur\'an graph $T_2(n)$. It is natural to conjecture that for every $k$ and $n$, the value of $f(n,t_k(n)+1)$ is determined by $T_k(n) + e$, which is the graph obtained by adding an edge to a largest class of $T_k(n)$. It is not hard to check that the largest chordal subgraph of $T_k(n)+e$ has $kn - \lceil \frac{n}{k} \rceil + 2 - \binom{k+1}{2}$ edges. So we get the following conjecture. 
\begin{conjecture}\label{conj:+1}
	$f(n,t_k(n)+1) = kn - \lceil \frac{n}{k} \rceil + 2 - \binom{k+1}{2}$. 
\end{conjecture}
The authors of \cite{EGOZ} only studied Conjecture \ref{conj:+1} in the cases $k=2,3$, although they very likely had the full conjecture in mind. For $k=3$, they proved that $f(n,t_3(n)+1) \geq 7n/3-6$ and asked to determine $f(n,t_3(n)+1)$. This question was later mentioned again in the problem survey of Gy\'arf\'as \cite{Gyarfas}. 
Answering this question, we resolve Conjecture \ref{conj:+1} for the case $k=3$.
\begin{theorem}\label{thm:k=3}
	$f(n,t_3(n)+1) = 3n-\lceil \frac{n}{3} \rceil - 4$.
\end{theorem} 
Our next result proves Conjecture \ref{conj:+1} asymptotically for every $k$. In fact, we go a step further and determine $f(n,m)$ asymptotically for {\em every} value of $m$, answering the question of  Erd\H{o}s and Laskar.
\begin{theorem}\label{thm:general}
	Let $k,n \geq 1$ and $t_{k}(n)+1 \leq m \leq t_{k+1}(n)$. Set $a = m - t_k(n)$. Then $$
	f(n,m) = (k-1/k)n + \sqrt{2(k+1)a/k} - \binom{k+1}{2} - O(\sqrt{n}).
	$$ 
\end{theorem}

The construction giving the upper bound in Theorem \ref{thm:general} is to take an (unbalanced) complete $k$-partite graph with $k-1$ smaller classes of the same size and one bigger class, and to add a balanced complete bipartite graph inside the bigger class. One then needs to optimize the sizes of the classes and the size of the complete bipartite graph so as to minimize the size of chordal subgraphs. It is best to take the $k-1$ smaller classes of size $\frac{n-k}{r}$, the bigger class of size $\frac{n+(k-1)r}{k}$ and the complete bipartite of size $r \times r$, where $r := \sqrt{\frac{2ka}{k+1}}$. See Figure 1, and see Section \ref{sec:general} for the details.

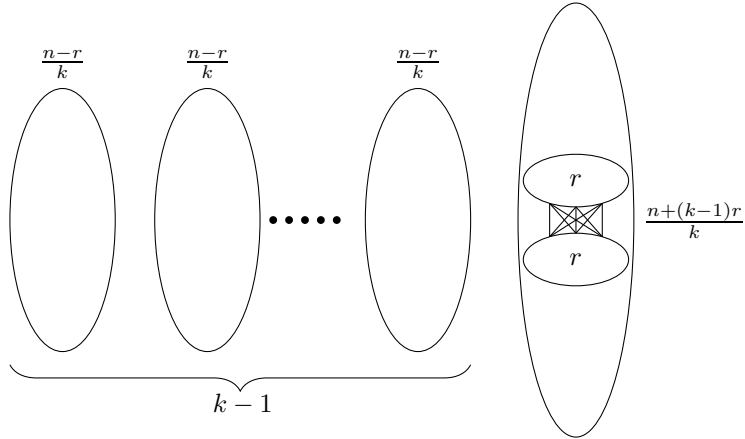
\begin{figure}\label{fig:construction}
    \centering
    \begin{tikzpicture}[scale = 0.7]
\draw (-2.75,0) ellipse (1 and 2.5);
\draw (0,0) ellipse (1 and 2.5);
\draw (4,0) ellipse (1 and 2.5);
\draw[dotted, line width = 3, line cap = round, dash pattern=on 0pt off 2\pgflinewidth] (1.25,0) -- (2.75,0); 
\draw (7,0) ellipse (1*1.1 and 3.75*1.1);
\draw (7,0.75) ellipse (1 and 0.5);
\draw (7,-0.75) ellipse (1 and 0.5);
\draw [decorate, decoration = {brace,amplitude=10pt}] (5,-2.75) --  (-3.75,-2.75);
\draw (0.675,-3.5) node {$k-1$};
\coordinate (x1) at ({7+1*cos(-60)},{0.75+0.5*sin(-60)});
\coordinate (x2) at ({7+1*cos(-120)},{0.75+0.5*sin(-120)});
\coordinate (x3) at ({7+1*cos(-90)},{0.75+0.5*sin(-90)});
\coordinate (y1) at ({7+1*cos(60)},{-0.75+0.5*sin(60)});
\coordinate (y2) at ({7+1*cos(120)},{-0.75+0.5*sin(120)});
\coordinate (y3) at ({7+1*cos(90)},{-0.75+0.5*sin(90)});
\foreach \i in {1,2,3}
{
\foreach \j in {1,2,3}
{
\draw (x\i) -- (y\j);
}
}
\draw (-2.75,3) node {$\frac{n-r}{k}$};
\draw (0,3) node {$\frac{n-r}{k}$};
\draw (4,3) node {$\frac{n-r}{k}$};
\draw (9.25,0) node {$\frac{n+(k-1)r}{k}$};
\draw (7,0.75) node {$r$};
\draw (7,-0.75) node {$r$};
\end{tikzpicture}
\caption{The construction showing optimality of Theorem \ref{thm:general}}
\end{figure}

For $k = 1,2$, we can go a step further and determine $f(n,m)$ exactly. This is done in the following two theorems. 
\begin{theorem}\label{thm:k=1}
	Let $n \geq 1$ and $m \leq t_2(n)$. Then $f(n,m) = \min\{r: t_2(r) \geq m\} - 1$. 
	
\end{theorem}
For $n \geq 1$ and $m \geq t_2(n)+1$, let $g_2(n,m)$ be the minimum of $2n-t+r$, taken over all pairs $t,r \geq 0$ satisfying $t(n-t) + t_2(r) \geq m$.
\begin{theorem}\label{thm:k=2}
	Let $n \geq 1$ and $t_{2}(n)+1 \leq m \leq t_{3}(n)$. Then $f(n,m) = g_2(n,m)-3$. 
\end{theorem}

\noindent
The extremal construction for Theorem \ref{thm:k=2} is given by taking a $t \times (n-t)$ complete bipartite graph and placing a complete bipartite graph with $r$ vertices inside the side of size $t$. 

\subsection{Proof ideas}\label{subsec:sketch}
Recall that a graph is chordal if and only if it can be obtained from the one-vertex graph by repeatedly adding {\em simplicial vertices}, i.e. vertices whose neighbourhood is a clique. In particular, adding simplicial vertices to a chordal graph keeps it chordal. We will often use this fact (implicitly) to claim that certain graphs are chordal.   

Let us first recall the argument used by Erd\H{o}s, Gy\'arf\'as, Ordman and Zalcstein \cite{EGOZ} to prove Conjecture \ref{conj:+1} for $k=2$ (and $n$ even). Let $G$ be a graph with $n$ vertices and $n^2/4+1$ edges, and let $x,y,z$ be a triangle in $G$. We need to show that $G$ has a chordal subgraph $H$ with at least $3n/2-1$ edges.
If $d(x) + d(y) + d(z) \geq 3n/2+2$, then take $H$ to be the subgraph consisting of all edges touching $x,y,z$. Suppose now that $d(x) + d(y) + d(z) \leq 3n/2+1$. Then by averaging, we can assume without loss of generality that $d(x) + d(y) \leq n$. Deleting $x,y$, we get a graph with at least $n^2/4+1 - (n-1) \geq (n-2)^2/4+1$ edges. By induction, this graph contains a chordal subgraph $H'$ with at least $3(n-2)/2 - 1$ edges. Adding the edges $xy,xz,yz$ gives the required chordal subgraph $H$. 

Our proof of Theorem \ref{thm:k=2} is also based on this inductive argument, but with two key differences. First, we need a relation between $g_2(n,m)$ and $g_2(n',m')$ (for $n' = n-2$, say), so that the induction can be carried through when deleting vertices. And second, it turns out that the induction scheme of deleting two vertices does not work to give the correct bound on $f(n,m)$ for all $m$ in the range of Theorem \ref{thm:k=2}. Instead, we sometimes need to delete just one vertex and then add two edges when adding the vertex back. To this end, we need to know that the deleted vertex has two neighbours which form an edge in the chordal subgraph $H'$ that we find using induction. To guarantee this, we strengthen the induction hypothesis to say that not only does $G$ contain a chordal subgraph with the correct number of edges, but that any given triangle in $G$ can be included in such a chordal subgraph.

The idea of strengthening the induction hypothesis is also used in the proof of Theorem \ref{thm:k=3}. Here we show that every $K_4$ can be included in a chordal subgraph with the correct number of edges. This proof has a more involved case analysis. It would be interesting to find a shorter proof. 

% The proof of Theorem \ref{thm:general} also uses induction. Here, instead of deleting only a few vertices, we delete a large number of vertices. More precisely, we find
% a $(k-1)$-clique $x_1,\dots,x_{k-1}$ and a forest $F$ inside $N(x_1,\dots,x_{k-1})$ such that $F$ has few components. We then delete $V(F)$ and $x_2,\dots,x_{k-1}$ and apply induction to find a chordal subgraph $H'$. We then add to $H'$ the edges of the clique $x_1,\dots,x_{k-1}$, the edges of $F$, and the edges between $V(F)$ and $x_1,\dots,x_{k-1}$. Observe that when adding back the vertices of $F$ one by one, most vertices add $k$ edges: one edge in $F$ and $k-1$ edges to $x_1,\dots,x_{k-1}$ (this fails once for each connected component of $F$, and this is why we want the number of components to be small). On the other hand, the main term in Theorem \ref{thm:general} is $(k-1/k)n$, which suggests that each vertex adds $k-1/k$ edges on average. So we see that we are gaining over the required bound (at least if we ignore the second term $\sqrt{2(k+1)a/k}$ for the moment). A somewhat lengthy calculation shows that this argument indeed works for any value of $a$. 

The proof of Theorem \ref{thm:general} is based on induction as well. Here, instead of deleting only a few vertices, we delete a large number of vertices. To give the general idea, we sketch first the proof in the case $m = t_3(n) + 1$. So let $G$ be a graph with $n$ vertices $t_3(n)+1$ edges. We need to show that $G$ has a chordal subgraph $H$ with at least $\frac{8n}{3} - 6 - C\sqrt{n}$ edges.  
Let us assume first that $e(G) \geq t_3(n) + 2n$. By a theorem of Faudree \cite{Faudree92} (see also \cite{Edwards76,BN}), there is a triangle $x,y,z \in V(G)$ with $d(x) + d(y) + d(z) \geq 6e(G)/n \geq 2n+12$. In particular, $x,y,z$ have at least $12$ common neighbours. Let $w_1,\dots,w_7$ be seven of them. If $d(x)+d(y)+d(z)+d(w_i) \geq \frac{8n}{3} - C\sqrt{n}$ for some $i$, then take $H$ to be the subgraph consisting of all edges touching $x,y,z,w_i$. This $H$ is chordal and $e(H) = d(x)+d(y)+d(z)+d(w_i) - 6$, so we are done. Suppose then that $d(x)+d(y)+d(z)+d(w_i) \leq \frac{8n}{3} - C\sqrt{n}$ for every $i$. In particular, $d(w_i) \leq \frac{2n}{3} - C\sqrt{n}$. Assume that $d(x) \geq d(y) \geq d(z)$, so that $d(x) \geq \frac{2n}{3}$ and hence $d(y) + d(z) + d(w_i) \leq 2n-C\sqrt{n}$ for each $i$. Delete $y,z,w_1,\dots,w_7$ to get a graph $G'$ on $n-9$ vertices. It is easy to see that $e(G')$ is well above $t_3(n-9)+1$. So by the induction hypothesis, there exists a chordal subgraph $H'$ of $G'$ with $e(H') \geq \frac{8(n-9)}{3} - 6 - C\sqrt{n}$ edges. Now add back the vertices $y,z,w_1,\dots,w_7$, and add to $H'$ the edges of the triangle $x,y,z$ and the edges between $x,y,z$ and $w_1,\dots,w_7$. This is a total of $24$ edges. So $e(H) = e(H') + 24 \geq \frac{8n}{3} - 6 - C\sqrt{n}$, as required. It is also easy to see that $H$ is chordal (if we add the new vertices in the order $y,z,w_1,\dots,w_7$, then we always add a simplicial vertex). The number $7$ was chosen here so that the number of edges added would be large enough for the induction to carry through. But the key point is that such a number must exist. Indeed, each $w_i$ contributes $3$ edges to $H$. On the other hand, the term $\frac{8n}{3}$ suggests that it is enough to add $\frac{8}{3}$ edges per vertex on average. So by adding $3$ edges per vertex, we are gaining over the required bound. 

It now remains to handle the case that $e(G) \leq t_3(n) + 2n$. Here we proceed as follows. If the minimum degree of the graph is at least $\frac{2n}{3} - \sqrt{n}$, then take a 4-clique $x,y,z,w$ and take $H$ to be the subgraph consisting of edges touching $x,y,z,w$. Else, delete a vertex of minimum degree and continue with the remaining graph. After $O(\sqrt{n})$ steps, we get a graph with $n' = n - O(\sqrt{n})$ vertices and at least $t_3(n') + 2n'$ edges, so we can apply the first case. 

To prove the general case of Theorem \ref{thm:general}
we find a $(k-1)$-clique $x_1,\dots,x_{k-1}$ and a forest $F$ inside $N(x_1,\dots,x_{k-1})$ such that $F$ has few components. We delete $V(F)$ and $x_2,\dots,x_{k-1}$ and apply induction to find a chordal subgraph $H'$. We then add to $H'$ the edges of the clique $x_1,\dots,x_{k-1}$, the edges of $F$, and the edges between $V(F)$ and $x_1,\dots,x_{k-1}$. Note that when adding back the vertices of $F$ one by one, most vertices contribute $k$ edges: one edge in $F$ and $k-1$ edges to $x_1,\dots,x_{k-1}$ (this fails once for each connected component of $F$, and this is why we want the number of components to be small). On the other hand, the main term in Theorem \ref{thm:general} is $(k-1/k)n$, which suggests that each vertex adds $k-1/k$ edges on average. So again we are gaining over the required bound (at least if we ignore the second term $\sqrt{2(k+1)a/k}$ for the moment). A somewhat lengthy calculation shows that this argument indeed works for any value of $a$. 

\vspace{0.25cm}
The rest of this short paper is organized as follows. 
Theorem \ref{thm:general} is proved in Section \ref{sec:general}, Theorems \ref{thm:k=1}-\ref{thm:k=2} in Section \ref{sec:k=1,2}, and Theorem \ref{thm:k=3} in Section \ref{sec:k=3}.

\section{Proof of Theorem \ref{thm:general}}\label{sec:general}
In this section we prove Theorem \ref{thm:general}.
We begin with the upper bound. Here we use the following construction. For simplicity, assume that $n$ is divisible by $k,k+1$. For general $n$ the construction is essentially the same (and, since we are only interested in an approximate result, we are allowed a small error due to divisibility issues). Fix $k\geq 1$ and $m \leq t_{k+1}(n) = \frac{kn^2}{2(k+1)}$, so that $a := m - t_k(n) \leq \frac{kn^2}{2(k+1)} - \frac{(k-1)n^2}{2k} = \frac{n^2}{2k(k+1)}$.
Set $r := \sqrt{\frac{2ka}{k+1}} \leq \frac{n}{k+1}$. 
Consider a complete $k$-partite graph with sides $X,Y_1,\dots,Y_{k-1}$ such that
$|X| = \frac{n + (k-1)r}{k}$ and $|Y_i| = \frac{n-r}{k}$ for every $1 \leq i \leq k-1$. Place an $r \times r$ complete bipartite graph with sides $A,B$ inside $X$. This is possible as $2r \leq \frac{n + (k-1)r}{k}$. The resulting graph $G$ has 
\begin{eqnarray*}
e(G) &=& (k-1) \cdot \frac{n + (k-1)r}{k} \cdot \frac{n-r}{k} + \binom{k-1}{2} \left( \frac{n-r}{k} \right)^2 + r^2 = \frac{(k-1)n^2}{2k} + \frac{(k+1)r^2}{2k}\\
&=& \frac{(k-1)n^2}{2k}+a=t_k(n)+a=m.
\end{eqnarray*}
Let $H$ be a chordal subgraph of $G$. We have $e_H(A,B) \leq |A| + |B| - 1 = 2r-1$, $e_H(A,Y_i) \leq |A| + |Y_i| - 1$, $e_H(X\setminus A,Y_i) \leq |X| - |A| + |Y_i| - 1$ and $e_H(Y_i,Y_j) \leq |Y_i| + |Y_j| - 1$, because each of these bipartite graphs is induced in $G$, so its intersection with $H$ is a forest. So 
\begin{align*}
e(H) &\leq 2r-1 + (k-1)|X| - 2(k-1) + 2\sum_{i = 1}^{k-1}{|Y_i|} + \sum_{1 \leq i < j \leq k-1}(|Y_i| + |Y_j| - 1) \\ &= 
kn - |X| + 2r - \binom{k+1}{2} = 
(k-1/k)n + \frac{(k+1)r}{k} - \binom{k+1}{2} \\ &= (k-1/k)n + \sqrt{2(k+1)a/k} - \binom{k+1}{2},
\end{align*}
giving the upper bound on $f(n,m)$ for Theorem \ref{thm:general}. We now prove the lower bound, which we restate for convenience as follows. 
%From now on, we prove the lower bound. For $k=1$ this follows from Theorem \ref{thm:k=1}, so we assume that $k \geq 2$. 

\begin{theorem}\label{thm:approximate}
	For every $k \geq 1$ there is $C = C(k)$ such that the following holds. Let $n,a \geq 1$, and let $G$ be a graph with $n$ vertices and at least $t_k(n) + a$ edges. Then $G$ has a chordal subgraph with at least $(k-1/k)n + \sqrt{2(k+1)a/k} - C\sqrt{n} - \binom{k+1}{2}$ edges. 
\end{theorem}

For the proof of Theorem \ref{thm:approximate} we need two lemmas. The following lemma uses an argument originally used by Edwards \cite{Edwards76,Edwards77} and Faudree \cite{Faudree92} (see also \cite{BN}) to find cliques with a large degree sum.  

\begin{lemma}\label{lem:k-1 clique}
	Let $k,n,a \geq 1$ and let $G$ be a graph with $n$ vertices and at least
	$\frac{(k-1)n^2}{2k} + a$ edges.
	Consider the following process: for $i =1,2,\dots$, take $x_i$ to be a vertex of maximum degree among all vertices in $N(x_1,\dots,x_{i-1})$. Then this process continues for at least $k$ steps, and $N(x_1,\dots,x_{k-1})$ contains at least $a$ \nolinebreak edges.
\end{lemma} 
\begin{proof}
	We prove the lemma by induction on $k$. The base case $k=1$ is trivial. Let $k \geq 2$. By the induction hypothesis, the process continues for at least $k-1$ steps. It remains to show that $N(x_1,\dots,x_{k-1})$ contains at least $a$ edges, because this would also imply that $N(x_1,\dots,x_{k-1}) \neq \emptyset$ and hence the process continues for at least $k$ steps. 
	For $1 \leq i \leq k-1$, let $S_i$ be the set of vertices which are adjacent to $x_1,\dots,x_{i-1}$ but not adjacent to $x_i$. In particular, $S_1$ is just the set of vertices not adjacent to $x_1$ and $x_i \in S_i$ for all $i$. Then $V(G) = S_1 \cup \dots \cup S_{k-1} \cup N(x_1,\dots,x_{k-1})$. 
	%	Therefore, the number of edges not contained in $N(x_1,\dots,x_{k-1})$ is at most
	%	$$
	%		\sum_{i=1}^{k-1}{|A_i| \cdot d_i} \leq \sum_{i=1}^{k-1}{(n-d_i)d_i}
	%	$$
	Put $S := S_1 \cup \dots \cup S_{k-1}$, $N := N(x_1,\dots,x_{k-1})$, $s_i := |S_i|$, $s = |S|$ and $d_i := d(x_i)$. Note that $s_i \leq n - d_i$. 
	Also, all vertices in $S_i$ have degree at most $d_i$. We have
	\begin{equation}\label{eq:degree sum in A}
	e(N,S) + 2e(S) = \sum_{v \in S}d(v) \leq \sum_{i = 1}^{k-1}{s_i \cdot d_i} \leq \sum_{i=1}^{k-1}{s_i(n-s_i)}.
	\end{equation}
	Since $e(G) = e(N) + e(S) + e(N,S)$, we have $2e(S) = 2e(G) - 2e(N) - 2e(N,S)$. 
	Plugging this into \eqref{eq:degree sum in A} and rearranging, we get
	\begin{equation}\label{eq:e(N)}
	e(N) \geq e(G) - \frac{1}{2}e(N,S) - \frac{1}{2} \cdot \sum_{i = 1}^{k-1}{s_i(n-s_i)}.
	\end{equation}
	We have 
	$e(N,S) \leq |N| \cdot |S| = (n - s)s$. Also, by Cauchy-Schwarz, 
	$$\sum_{i = 1}^{k-1}{s_i(n-s_i)} = ns - \sum_{i=1}^{k-1}{s_i^2} \leq ns - \frac{s^2}{k-1}.
	$$
	Plugging this into \eqref{eq:e(N)} gives
	$$
	e(N) \geq e(G) - \frac{1}{2}(n-s)s - \frac{1}{2}\left( ns - \frac{s^2}{k-1} \right) = e(G) - ns + \frac{ks^2}{2(k-1)}. 
	$$
	The maximum of $ns - \frac{ks^2}{2(k-1)}$ is obtained at $s = \frac{(k-1)n}{k}$ and equals $\frac{(k-1)n^2}{2k}$. Hence, $e(N) \geq e(G) - \frac{(k-1)n^2}{2k} \geq a$, as required.   
\end{proof}
%	\begin{lemma}\label{lem:star or linear forest}
%		Let $G$ be a graph with $n$ vertices and $r \geq n$ edges. Then $G$ contains a forest $F$ with at least $\sqrt{r/2}$ vertices and at least $v(F) - \frac{3n}{\sqrt{r}}$ edges. 
%	\end{lemma}
	\begin{lemma}\label{lem:star or linear forest}
		Let $G$ be a graph with $n$ vertices and $a$ edges. Let $s \geq 1$ and suppose that $a \geq 2s^2$. Then $G$ contains a forest $F$ with $s$ vertices and at least $s - 1 - \frac{sn}{a}$ edges. 
	\end{lemma}
	\begin{proof}
		Let $C_1,\dots,C_m$ be the connected components of $G$ with $|C_1| \geq \dots \geq |C_m|$. Let $\ell \geq 1$ be the minimal integer satisfying $|C_1| + \dots + |C_{\ell}| \geq s$. 
		If $\ell \leq 1 + \frac{sn}{a}$ then take $F$ to be a forest contained in $C_1 \cup \dots \cup C_{\ell}$ having $s$ vertices and $\ell$ connected components. Suppose now by contradiction that $\ell > 1 + \frac{sn}{a}$.
		Set $r = |C_1| + \dots + |C_{\ell-1}|$. 
		Then $r < s$ and $|C_{\ell-1}| \leq \frac{r}{\ell-1}$. We have $e(G) \leq \binom{r}{2} + \sum_{i = \ell}^{m}{\binom{|C_i|}{2}}$.
		By convexity, the sum $\sum_{i = \ell}^{m}{\binom{|C_i|}{2}}$ is maximized when all except maybe one of the $|C_i|$'s are equal to their maximal value, which is $\frac{r}{\ell-1}$. So
		\begin{align*}
		    e(G) \leq \binom{r}{2} + \left\lceil \frac{n-r}{r/(\ell-1)} \right\rceil \cdot \binom{r/(\ell-1)}{2} &\leq 
		\binom{r}{2} + \left( 1 + \frac{n-r}{r/(\ell-1)} \right) \cdot \binom{r/(\ell-1)}{2} \\&\leq 
		\binom{r}{2} + \binom{r/(\ell-1)}{2} + \frac{nr}{2(\ell-1)} \\&<
		2\binom{s}{2} + \frac{sn}{2sn/a} \leq a,
		\end{align*}
		where the last inequality uses $a \geq 2s^2$. We got a contradiction to $e(G) = a$. 
	\end{proof}
%	\begin{theorem}
%		Let $k \geq 2$, $n \geq 1$ and $1 \leq r \leq \frac{kn^2}{2(k+1)} - \frac{(k-1)n^2}{2k} = \frac{n^2}{2k(k+1)}$. Let $G$ be a graph with $n$ vertices and $\frac{(k-1)n^2}{2k} + r$ edges. Then $G$ has a chordal subgraph with at least $(k-1/k)n + \sqrt{2(k+1)r/k} - C\sqrt{n} - \binom{k+1}{2}$ edges. 
%	\end{theorem}
\noindent
We are now ready to prove Theorem~\ref{thm:approximate}. An overview of the proof can be found in Section~\ref{subsec:sketch}.

	\begin{proof}[Proof of Theorem \ref{thm:approximate}]
	    The proof is by induction on $n$.
		Fix constants $k \ll c \ll c_1 \ll C$, to be chosen implicitly later. 
		Suppose first that $a \leq (ck + 1)^2n$. In this case we proceed as follows. If $\delta(G) \geq \lfloor \frac{(k-1)n}{k} \rfloor - c_1\sqrt{n}$, then take a $(k+1)$-clique $x_1,\dots,x_{k+1} \in V(G)$ and take $H$ to consist of all edges that touch $x_1,\dots,x_{k+1}$. Then $H$ is chordal and 
%		$e(H) = \sum_{i=1}^{k+1}{d(x_i)} - \binom{k+1}{2} \geq (k+1) \cdot \left( \frac{(k-1)n}{k} - c_1\sqrt{n} \right) - \binom{k+1}{2} = (k-1/k)n - (k+1)c_1\sqrt{n} - \binom{k+1}{2} \geq (k-1/k)n + \sqrt{2(k+1)r/k} - C\sqrt{n} - \binom{k+1}{2}$, 
		\begin{align*}
		e(H) &= \sum_{i=1}^{k+1}{d(x_i)} - \binom{k+1}{2} \geq (k+1) \cdot \left( \frac{(k-1)n}{k} - 2c_1\sqrt{n} \right) - \binom{k+1}{2} \\ &= (k-1/k)n - 2(k+1)c_1\sqrt{n} - \binom{k+1}{2} \geq (k-1/k)n + \sqrt{2(k+1)a/k} - C\sqrt{n} - \binom{k+1}{2},
		\end{align*}
		where the last inequality holds as $C \gg c_1,c$ and $a \leq (ck + 1)^2n$. Suppose now that there is $v \in V(G)$ with $d(v) \leq \lfloor \frac{(k-1)n}{k}\rfloor - c_1\sqrt{n}$. Let $G' = G - v$. Then 
        $$e(G') \geq t_k(n) + a - \left\lfloor \frac{(k-1)n}{k}\right\rfloor + c_1\sqrt{n} 
		= t_k(n-1) + a + c_1\sqrt{n}.
        $$
        By the induction hypothesis with parameter $a' = a + c_1\sqrt{n}$, $G'$ contains a chordal subgraph $H'$ with 
		$$
        e(H') \geq (k - 1/k)(n-1) + \sqrt{2(k+1)a'/k} - C\sqrt{n} - \binom{k+1}{2}.$$ 
        As $(k-1/k)(n-1) \geq (k-1/k)n - k$, it suffices to show that $\sqrt{2(k+1)a'/k} \geq \sqrt{2(k+1)a/k} + k$. Squaring and plugging in the value of $a'$, we get 
        $$2(k+1)/k \cdot (a+c_1\sqrt{n}) = 2(k+1)a'/k \geq 2(k+1)a/k + 2k\sqrt{2(k+1)a/k} + k^2.$$ Cancelling the term $2(k+1)a/k$ from both sides and rearranging, we see that it is enough to have 
        $$c_1\sqrt{n} \geq  \frac{k^2}{k+1}\sqrt{2(k+1)a/k} + \frac{k^3}{2(k+1)},$$ which holds because $a \leq (ck + 1)^2n$ and $c_1 \gg c$. 
		
		For the rest of the proof we assume that $a \geq (ck + 1)^2n$. 
		Note that 
		$e(G) \geq t_k(n) + a \geq \frac{(k-1)n^2}{2k} + \frac{a}{2}$ because $t_k(n) \geq \frac{(k-1)n^2}{2k} - O_k(1)$ and $a \geq c \gg k$.
		Let $x_1,\dots,x_{k-1}$ be as in Lemma \ref{lem:k-1 clique} and put $N = N(x_1,\dots,x_{k-1})$. By Lemma \ref{lem:k-1 clique} we have $e(N) \geq \frac{a}{2}$. Also, the choice of $x_1,\dots,x_{k-1}$ in Lemma \ref{lem:k-1 clique} implies that
		$d(y) \leq d(x_{k-1}) \leq \dots \leq d(x_1)$ for every $y \in N$. For convenience, we set
		$$
		d_0 := \frac{(k-1)n}{k} + \sqrt{\frac{2a}{k(k+1)}} - c\sqrt{n}.
		$$
    Note that $d_0 \geq \frac{(k-1)n}{k}$ by our assumption that $a \geq (ck + 1)^2n$. 
		\begin{claim}\label{claim:deletion forest}
			If the statement of the theorem does not hold, then $G[N]$ contains a forest $F$ with $v(F) = \lfloor \sqrt{n} \rfloor$, $e(F)\geq v(F) - 1 - \frac{2n^{3/2}}{a}$, and 
			\begin{equation}\label{eq:forest degree bound}
			\sum_{y \in V(F)}{d(y)} \leq v(F) \cdot d_0 + n/k.
			\end{equation}  
		\end{claim}
		\begin{proof} 
			We consider two cases. Suppose first that there is $x_k \in N$ such that $d(x_k) \geq d_0$. 
			Then $d(x_i) \geq d_0$ for every $1 \leq i \leq k-1$. Hence, $d(x_1) + \dots + d(x_k) \geq k \cdot d_0$. This means that $x_1,\dots,x_k$ have at least 
			$$
			kd_0 - (k-1)n = \sqrt{\frac{2ka}{(k+1)}} - ck\sqrt{n} \geq 
			\sqrt{a} - ck\sqrt{n} \geq \sqrt{n}
			$$
			common neighbours, where the last inequality holds by the assumption $a \geq (ck + 1)^2n$. 
			Take $F$ to be the star whose center is $x_k$ and whose leaves are $\lfloor \sqrt{n} \rfloor - 1$ common neighbours of $x_1,\dots,x_k$.  
			Let $y \in N(x_1,\dots,x_k)$. If $d(y) \geq d_0$ then 
			$$
			d(x_1) + \dots + d(x_{k}) + d(y) \geq (k+1)d_0 = (k-1/k)n + \sqrt{2(k+1)a/k} - (k+1)c\sqrt{n},
			$$
			and then the subgraph consisting of all edges touching $\{x_1,\dots,x_k,y\}$ is a chordal graph with at least $(k- \nolinebreak 1/k)n + \sqrt{2(k+1)a/k} - (k+1)c\sqrt{n} - \binom{k+1}{2}$ edges, so the assertion of the theorem holds. Hence, we may assume that $d(y) \leq d_0$ for every $y \in N(x_1,\dots,x_k)$. This means that 
			$$
			\sum_{v \in V(F)}{d(v)} \leq d(x_k) + (v(F) - 1) \cdot d_0 \leq n/k + v(F) \cdot d_0,
			$$
			as required by the claim. Also, $F$ has the right number of edges, as $e(F)=v(F)-1$.
			
			The second case is that $d(y) \leq d_0$ for every $y \in N$. Since $e(N) \geq \frac{a}{2} \geq 2n$, we can apply Lemma \ref{lem:star or linear forest} to $G[N]$ with parameters $\frac{a}{2}$ and $s = \lfloor \sqrt{n} \rfloor$ to obtain a forest $F$ with $\lfloor \sqrt{n} \rfloor$ vertices and at least $v(F) - 1 - \frac{2n^{3/2}}{a}$ edges. All vertices in $F$ have degree at most $d_0$, so \eqref{eq:forest degree bound} holds. 
		\end{proof}
		We continue with the proof of the theorem. Let $F$ be the forest given by Claim \ref{claim:deletion forest}. Let $G'$ be the graph obtained from $G$ by deleting the $t := k-2 + v(F)$ vertices $T := \{ x_2,\dots,x_{k-1} \} \cup V(F)$. By \eqref{eq:forest degree bound}, we have 
		$$
		\sum_{v \in T}{d(v)} \leq d(x_2) + \dots + d(x_{k-1}) + v(F) \cdot d_0 + n/k \leq (v(F) + k - 2) \cdot d_0 + n = t \cdot d_0 + n,
		$$  
        where the second inequality uses that $d_0 \geq \frac{(k-1)n}{k}$.
		As $e(G) \geq t_k(n) + a \geq \frac{(k-1)n^2}{2k} - O_k(1) + a$, we have that
		\begin{align*}
		e(G') \geq e(G) - t \cdot d_0 - n &\geq \frac{(k-1)n^2}{2k}- O_k(1) + a - t \cdot d_0 - n  \\ &= 
		\frac{(k-1)(n-t)^2}{2k} + \frac{(k-1)nt}{k} - \frac{(k-1)t^2}{2k} - O_k(1) + a - t \cdot d_0 - n \\ &=   
		\frac{(k-1)(n-t)^2}{2k} - O_k(1) + a - \frac{(k-1)t^2}{2k} - 
		t \cdot \sqrt{\frac{2a}{k(k+1)}} + ct\sqrt{n} - n \\ &\geq
		t_k(n-t) + a - \frac{(k-1)t^2}{2k} - 
		t \cdot \sqrt{\frac{2a}{k(k+1)}} + \frac{c}{2}t\sqrt{n},
		\end{align*}
		where the last inequality uses that $t \geq \lfloor \sqrt{n} \rfloor$ and $c \gg k$, so that $\frac{c}{2}t\sqrt{n} \geq O_k(1) + n$. 
		Set 
		\begin{equation*}\label{eq:r'}
		a' := a - \frac{(k-1)t^2}{2k} - t \cdot \sqrt{\frac{2a}{k(k+1)}} + \frac{c}{2}t\sqrt{n},
		\end{equation*}
		so that $e(G') \geq t_k(n-t)+a'$. 
		We have $a' \geq 1$ because $t \leq \sqrt{n}+k-2$, $a \geq c^2n$ (say) and $c \gg k$. By the induction hypothesis, $G'$ contains a chordal subgraph $H'$ of size at least
%		\begin{equation}\label{eq:H'}
%		e(H') \geq (k-1/k) \cdot (n-t) +  \sqrt{2(k+1)r'/k} - C\sqrt{n-t} - \binom{k+1}{2}.
%		\end{equation}
		$$
		e(H') \geq (k-1/k) \cdot (n-t) +  \sqrt{2(k+1)a'/k} - C\sqrt{n-t} - \binom{k+1}{2}.
		$$
		Let $H$ be the subgraph of $G$ obtained by adding to $H'$ the edges of the clique $x_1,\dots,x_{k-1}$, the edges between $x_1,\dots,x_{k-1}$ and $V(F)$, and the edges of $F$. Then $H$ is chordal.
		To complete the proof, it suffices to verify \nolinebreak that 
        $$e(H) \geq (k-1/k)n + \sqrt{2(k+1)a/k} - C\sqrt{n} - \binom{k+1}{2}.$$
%		\begin{equation}\label{eq:edges of H}
%		e(H) \geq (k-1/k)n + \sqrt{2(k+1)a/k} - C\sqrt{n} - \binom{k+1}{2}.
%		\end{equation}
		By the definition of $H$, we have 
        $$e(H) = e(H') + \binom{k-1}{2} + (k-1) \cdot v(F) + e(F).$$
        Note that 
        $$(k-1) \cdot v(F) + e(F) \geq k \cdot v(F) - 1 - \frac{2n^{3/2}}{a} = k \cdot (t - k + 2) - 1 - \frac{2n^{3/2}}{a},$$ and therefore
		$$\binom{k-1}{2} + (k-1) \cdot v(F) + e(F) \geq tk - \frac{(k+1)(k-2)}{2} - 1  - \frac{2n^{3/2}}{a}.$$ For convenience, set $h := \frac{(k+1)(k-2)}{2} + 1  + \frac{2n^{3/2}}{a}$. Then
		\begin{equation}\label{eq:H}
		e(H) \geq (k-1/k) \cdot (n-t) +  \sqrt{2(k+1)a'/k} - C\sqrt{n-t} - \binom{k+1}{2} + tk - h.
		\end{equation}
		So it remains to verify that the right-hand side of \eqref{eq:H} is at least as large as 
        $$(k-1/k)n + \sqrt{2(k+1)a/k} - C\sqrt{n} - \binom{k+1}{2}.$$ 
        Cancel the terms $(k-1/k)n$, $\binom{k+1}{2}$ which appear in both expressions. Also, we may drop the terms $C\sqrt{n-t}$, $C\sqrt{n}$. After rearranging, we get the inequality
		$$
		\sqrt{2(k+1)a'/k} \geq \sqrt{2(k+1)a/k} - \frac{t}{k} + h.
		$$
		By squaring and plugging in the value of $a'$, we get:
%		\begin{align*}
%		e(H) &\geq (k-1/k)n + \frac{t}{k} - k(k-2) + \binom{k-1}{2} - \frac{n}{\sqrt{r}} + \sqrt{2(k+1)r'/k} - C\sqrt{n-t} - \binom{k+1}{2} \\ &\geq 
%%		(k-1/k)n + \frac{t}{k} - \frac{(k-2)(k+1)}{2} - \frac{n}{\sqrt{r}} + \sqrt{2(k+1)r'/k} -  C\sqrt{n} - \binom{k+1}{2}.
%		(k-1/k)n + \frac{t}{k} - \frac{k^2n}{\sqrt{r}} + \sqrt{2(k+1)r'/k} -  C\sqrt{n} - \binom{k+1}{2}.
%		\end{align*}
%		\begin{equation*}%\label{eq:induction inequality}
%		\sqrt{2(k+1)r'/k} \geq \sqrt{2(k+1)r/k} - \frac{t}{k} + \frac{k^2n}{\sqrt{r}}.
%		\end{equation*}
%		By squaring and plugging in our choice of $r'$ in \eqref{eq:r'}, we get
		\begin{equation}\label{eq:induction inequality 2}
		\begin{split}
		& \frac{2(k+1)}{k} \cdot \left( a - \frac{(k-1)t^2}{2k} - t \cdot \sqrt{\frac{2a}{k(k+1)}} + \frac{c}{2}t\sqrt{n} \right) \geq \\ &\frac{2(k+1)a}{k} + \frac{t^2}{k^2} + h^2 - \frac{2t}{k} \cdot \sqrt{\frac{2(k+1)a}{k}} + 2h \cdot \sqrt{\frac{2(k+1)a}{k}} - \frac{2th}{k}.
		\end{split}
		\end{equation}
		Both sides of the inequality \eqref{eq:induction inequality 2} have the terms $\frac{2(k+1)}{k}a$ and $-\frac{2t}{k} \cdot \sqrt{\frac{2(k+1)a}{k}}$. We can also drop the negative term $\frac{2th}{k}$ on the right-hand side. After rearranging, we get the inequality 
		\begin{equation}\label{eq:last}
		\frac{2(k+1)}{k} \cdot \frac{c}{2}t\sqrt{n} \geq t^2 + h^2 + 2h \cdot \sqrt{\frac{2(k+1)a}{k}}.
		\end{equation}
		We have $t \leq \sqrt{n} + k$ and $h \leq O_k(1) + \frac{2n^{3/2}}{a} \leq O_k(1) + \sqrt{n}$, so $t^2,h^2 \leq O_k(n)$. Also, 
        $$h \sqrt{a} \leq \left(O_k(1) + \frac{2n^{3/2}}{a}\right) \cdot \sqrt{a} \leq O_k(n) + \frac{2n^{3/2}}{\sqrt{a}} = O_k(n),$$ as $n \leq a \leq n^2$. So the right-hand side of \eqref{eq:last} is $O_k(n)$. On the other hand, the left-hand side is larger than $\frac{cn}{2}$ because $t \geq \lfloor \sqrt{n}\rfloor \geq \sqrt{n}/2$. So \eqref{eq:last} holds because $c \gg k$, as required.  
	\end{proof}
	
\section{Proof of Theorems \ref{thm:k=1} and \ref{thm:k=2}}\label{sec:k=1,2}
\begin{proof}[Proof of Theorem~\ref{thm:k=1}]
For the upper bound, let $r \geq 1$ be the minimal integer satisfying $t_2(r) \geq m$, and take $G$ to be $T_2(r)$ with $n-r$ isolated vertices. Then $e(G) \geq m$, but every chordal subgraph of $G$ has at most $r-1$ edges. 
For the lower bound, we prove by induction on the number of vertices that every graph $G$ with $m$ edges has a chordal subgraph $H$ with at least $g_1(m)-1$ edges, where $g_1(m) := \min_{r: t_2(r) \geq m} r$. For $m = 0$, the assertion is trivial. Suppose $m \geq 1$ and let $xy \in E(G)$. 
Fix $r$ such that $g_1(m) = r$.
If $d(x) + d(y) \geq r$ then take $H$ to be the subgraph consisting of all edges touching $x,y$. This graph is chordal and has $d(x)+d(y)-1$ edges. Suppose now that $d(x) + d(y) \leq r-1$; without loss of generality, $d(x) \leq \lfloor \frac{r-1}{2} \rfloor$. Let $G' = G-x$. Then $e(G') = m - d(x) \geq m - \lfloor \frac{r-1}{2} \rfloor$. We claim that $e(G') > t_2(r-2)$. Indeed, if $e(G') \leq t_2(r-2)$ then $m \leq t_2(r-2) + \lfloor \frac{r-1}{2} \rfloor = t_2(r-1)$, in contradiction to the choice of $r$. So $g_1(e(G')) \geq r-1$. By the induction hypothesis, $G'$ contains a chordal subgraph $H'$ with at least $r-2$ edges. Now, $H' + \{xy\}$ is a chordal subgraph of $G$ with at least $r-1$ edges, as required.  
\end{proof}

The rest of this section is dedicated to proving Theorem~\ref{thm:k=2}. 
Recall that for $n \geq 1$ and $m \geq t_2(n) + 1$, we define 
$
g_2(n,m) := 
\min_{t,r}(2n-t+r),
$
where the minimum is taken over all integers $t,r \geq 0$ satisfying $t(n-t)+t_2(r) \geq m$.
% Observe that since $m \geq t_2(n)+1$, the condition $t(n-t)+t_2(r) \geq m$ requires that $r \geq 1$. Also, if $(t,r)$ achieve the minimum then $t \geq n/2$, because else we can replace $t$ with $n-t$, which decreases $2n-t+r$.  
We start by proving the upper bound in Theorem \ref{thm:k=2}.
First we claim that $g_2(n,m) \leq 2n$. 
Indeed, for $t = r = \lceil \frac{2n}{3} \rceil$ we have $t(n-t) + t_2(r) = t_3(n) \geq m$, so $g_2(n,m) \leq 2n - t + r = 2n$, as required.
Now take $t,r$ such that $t(n-t) + t_2(r) \geq m$ and $g_2(n,m) = 2n-t+r$. Since $g_2(n,m) \leq 2n$, we have $t \geq r$. Take a complete bipartite graph with sides $X$ of size $t$ and $Y$ of size $n-t$, and add a copy of $T_2(r)$ with sides $A,B$ inside $X$. The resulting graph $G$ has $t(n-t) + t_2(r) \geq m$ edges. Let $H$ be a chordal subgraph of $G$. Then $e_H(A,B) \leq |A|+|B|-1 = r-1$, $e_H(A,Y) \leq |A|+|Y|-1$ and $e_H(X\setminus A,Y) \leq |X| - |A| + |Y| - 1$, because each of these bipartite graphs is induced in $G$, so its intersection with $H$ is a forest. Overall, we got that $e(H) \leq r-1 + |X| + 2|Y| - 2 = 2n - t + r - 3 = g_2(n,m) - 3$. This shows that $f(n,m) \leq g_2(n,m)-3$, as required. To prove the lower bound in Theorem \ref{thm:k=2}, we prove the following stronger claim. 
%\begin{proposition}\label{prop:tightness}
%	For every $t_2(n) + 1 \leq m \leq t_3(n)$, there is a graph $G$ with $n$ vertices and $m$ edges such that every chordal subgraph of $G$ has at most $g(n,m)-3$ edges. 
%\end{proposition}
%\begin{proof}
%	First we claim that $f(n,m) \leq 2n$. It is not hard to check that one can choose $t,r$ such that $t = 2r$ and $t(n-t) + r^2 = t_3(n) \geq m$. Indeed, if $n \equiv 0,2 \pmod{3}$, take $t = \lceil \frac{2n}{3} \rceil$, $r = \lceil \frac{n}{3} \rceil$, and if $n \equiv 1 \pmod{3}$, take $t = \lfloor \frac{2n}{3} \rfloor$, $r = \lfloor \frac{n}{3} \rfloor$. Now, $f(n,m) \leq 2n - t + 2r = 2n$, as claimed.
%	
%	Now take $t,r$ such that $t(n-t) + r^2 \geq m$ and $f(n,m) = 2n-t+2r$. Since $f(n,m) \leq 2n$, we have $t \geq 2r$. Take a complete bipartite graph with sides $X$ of size $t$ and $Y$ of size $n-t$, and add an $r \times r$ complete bipartite graph with sides $A,B$ inside $X$. The resulting graph $G$ has $t(n-t) + r^2 \geq m$ edges. Let $H$ be a chordal subgraph of $G$. Then $e_H(A,B) \leq |A|+|B|-1 = 2r-1$, $e_H(A,Y) \leq |A|+|Y|-1$ and $e_H(X\setminus A,Y) \leq |X| - |A| + |Y| - 1$, because in each case we consider an induced bipartite subgraph of $G$, so the intersection of $H$ with this subgraph is a forest. Overall, $e(H) \leq 2r-1 + |X| + 2|Y| - 2 = 2n - t + 2r - 3 = f(n,m) - 3$, as required. 
%\end{proof}

\begin{theorem}\label{thm:k=2 main}
	Let $G$ be a graph with $n$ vertices and $m \geq t_2(n)+1$ edges, and let $x,y,z$ be a triangle in $G$. Then $G$ has a chordal subgraph with at least $g_2(n,m)-3$ edges which contains the edges $xy,xz,yz$.  
\end{theorem}

\noindent
We need the following facts about the numbers $g_2(n,m)$. 
\begin{lemma}\label{lem:range of t,r}
	In the definition of $g_2(n,m)$, we may assume that 
	% $2t-n \in \{ \lfloor \frac{r}{2} \rfloor, \lceil \frac{r}{2} \rceil \}$. 
	$-\frac{1}{2} \leq 2t - n - \frac{r}{2} \leq \frac{1}{2}$.
\end{lemma}
\begin{proof}
	Fix $t,r$ which achieve the minimum in the definition of $g_2(n,m)$; so $t(n-t)+t_2(r) \geq m$ and $g_2(n,m) = 2n-t+r$. Put $h(t,r) := 2t - n - \frac{r}{2}$. If $h(t,r) \in \{-\frac{1}{2},0,\frac{1}{2}\}$ then we are done. If not, we try replacing $(t,r)$ with $(t-1,r-1)$ or $(t+1,r+1)$.
	Suppose first that
	$h(t,r) \geq 1$. Replace $(t,r)$ with $(t-1,r-1)$. We have $2n - (t-1) + (r-1) = 2n-t+r$ and $(t-1)(n-t+1) + t_2(r-1) = t(n-t) + t_2(r) - n + 2t - \lfloor\frac{r}{2} \rfloor - 1 \geq t(n-t) + t_2(r) \geq m$, where the penultimate inequality uses $h(t,r) \geq 1$. So $(t-1,r-1)$ also achieves the minimum in the definition of $g_2(n,m)$. Moreover, $h(t-1,r-1) = 2(t-1) - n - \frac{r-1}{2} = h(t,r) - \frac{3}{2}$. So as long as $h(t,r) \geq 1$, we can replace $(t,r)$ with $(t-1,r-1)$ and decrease $h$ by $\frac{3}{2}$. At the last step, we decrease $h$ to be in $\{-\frac{1}{2}, 0, \frac{1}{2}\}$. 
	
	Similarly, suppose that $h(t,r) \leq -1$. Replace $(t,r)$ with $(t+1,r+1)$. We have $2n-(t+1) + (r+1) = 2n-t+r$ and $(t+1)(n-t-1) + t_2(r+1) = t(n-t) + t_2(r) + n - 2t + \lceil \frac{r}{2} \rceil - 1 \geq t(n-t) + t_2(r) \geq m$, where the penultimate inequality uses $h(t,r) \leq -1$. So $(t+1,r+1)$ also achieves the minimum in the definition of $g_2(n,m)$. Also, $h(t+1,r+1) = 2(t+1) - n - \frac{r+1}{2} = h(t,r) + \frac{3}{2}$. So as long as $h(t,r) \leq -1$, we can replace $(t,r)$ with $(t+1,r+1)$ and increase $h$ by $\frac{3}{2}$. At the last step, we increase $h$ to be in $\{-\frac{1}{2}, 0, \frac{1}{2}\}$. 
\end{proof}
\begin{lemma}\label{lem:-1}
	For $n \geq 1$ and $m \geq t_2(n) + 2$, it holds that $g_2(n,m-1) \geq g_2(n,m) - 1$.
\end{lemma}
\begin{proof}
	Let $t,r$ be such that $t(n-t) + t_2(r) \geq m-1$ and $2n-t+r = g_2(n,m-1)$. Since $m-1 \geq t_2(n)+1$ we have $r \geq 2$. Thus $t(n-t) + t_2(r+1) \geq m$, so $g_2(n,m) \leq 2n-t+(r+1) = g_2(n,m-1) + 1$, as required. 
\end{proof}
\begin{lemma}\label{lem:n-2}
	The following holds for every $n \geq 3$.
	\begin{enumerate}
		\item If $m \geq t_2(n) + 1$ then $m - n + 1 \geq t_2(n-2) + 1$ and $g_2(n-2,m-n+1) \geq g_2(n,m) - 3$.
		\item If $m \geq t_2(n) + 2$ then $m -n \geq t_2(n-2)+1$ and $g_2(n-2,m-n) \geq g_2(n,m) - 4$.
	\end{enumerate}
\end{lemma}
\begin{proof}
	The first part in both items follows from
	$t_2(n-2) = t_2(n) - n + 1$. Let $t,r$ such that $t(n-2-t) + t_2(r) \geq m-n+1$ and $g_2(n-2,m-n+1) = 2(n-2) - t + r$. We have 
	$(t+1)(n-1-t) + t_2(r) = t(n-2-t) + n-1 + t_2(r) \geq m$. Hence, $g_2(n,m) \leq 2n - (t+1) + r = g_2(n-2,m-n+1) + 3$. This proves the first item. For the second item, $g_2(n-2,m-n) \geq g_2(n-2,m-n+1) - 1 \geq g_2(n,m) - 4$, where the first equality uses Lemma \ref{lem:-1} since $m-n+1\geq t_2(n-2)+2$.
\end{proof} 
\begin{lemma}\label{lem:f(n,m) main}
	Let $n \geq 1$ and $m \geq t_2(n) + 1$, and let $d \geq 0$ be an integer satisfying $3d \leq g_2(n,m)-1$. Then $m-d \geq t_2(n-1)+1$ and $g_2(n-1,m-d) \geq g_2(n,m) - 2$. 
\end{lemma}

Lemma~\ref{lem:f(n,m) main} is used in the proof of Theorem~\ref{thm:k=2 main} in case one of the vertices $x,y,z$ has degree at most $d$. We then delete this vertex, apply induction, and then add the vertex back with its two incident edges on the triangle. 
The induction carries through thanks to the bound $g_2(n-1,m-d) \geq g_2(n,m) - 2$ from Lemma~\ref{lem:f(n,m) main}. The reason for the assumption $3d \leq g_2(n,m)-1$ is that if $d(x) + d(y) + d(z) \geq g_2(n,m)$ then the graph consisting of all edges touching $x,y,z$ is a chordal graph with at least $g_2(n,m)-3$ edges, as required by Theorem~\ref{thm:k=2 main}. So we may always assume that one of $x,y,z$ has degree $d$ with $3d \leq g_2(n,m)-1$.

% To prove this bound, we consider the pair $(t,r)$ achieving the minimum in the definition of $g_2(n,m)$ and the pair $(t',r')$ achieving the minimum in the definition of $g_2(n-1,m-d)$. We can upper-bound $m-d$ as a function of $t',r'$ (as $t'(n-1-t') + t_2(r') \geq m-d$) and also lower-bound $m$ as a function of $t,r$ (because the minimality of $(t,r)$ implies that $t(n-t)+t_2(r-1)<m$). Also, assuming by contradiction that $g_2(n-1,m-d) \leq g_2(n,m) - 3$ gives us a relation between $t,r$ and $t',r'$, and this allows us to replace $t',r'$ with $t,r$ in the upper bound for $m-d$. By combining these inequalities, we can get a lower bound for $d$ as a function of $t,r$. Analyzing this bound, we show that $3d \geq g_2(n,m)$, contradicting the assumption of the lemma. The detailed proof follows.   
\begin{proof}[Proof of Lemma~\ref{lem:f(n,m) main}]
	First we show that $m-d > t_2(n-1)$. Set $a := m - t_2(n)$. 
	Since $3d \leq g_2(n,m) - 1$, it is enough to show that 
	\begin{equation}\label{eq:f(n-1,m-d)}
	g_2(n,m) < 3(m - t_2(n-1)) + 1 = 3 \cdot \left\lfloor \frac{n}{2} \right\rfloor + 3a + 1.
	\end{equation}
	For $a=1$, $g_2(n,m) = n + \lfloor \frac{n}{2} \rfloor + 2 < 3 \cdot \lfloor \frac{n}{2} \rfloor + 4$, and the last expression equals the right-hand side of \eqref{eq:f(n-1,m-d)}. Suppose now that $a \geq 2$. 
	%	using that $t_2(n) - t_2(n-1) = \lfloor \frac{n}{2}\rfloor$.
	Set $k = \lceil \sqrt{\frac{a}{3}} \rceil$, so that $3k^2 \geq a$. Set $t := \lfloor \frac{n}{2} \rfloor + k$ and $r := 4k$. Then $t_2(r) = 4k^2$ and 
	$$t(n-t) = \left( \left\lfloor \frac{n}{2} \right\rfloor + k \right)\left( \left\lceil \frac{n}{2} \right\rceil - k \right) \geq t_2(n) - k^2,$$ so 
	% $t_2(n) + k\cdot\left( \lceil \frac{n}{2} \rceil - \lceil \frac{n}{2} \rceil \right) - k^2$, 
	$t(n-t) + t_2(r) \geq t_2(n) + 3k^2 \geq t_2(n) + a = m$. Hence, 
	$g_2(n,m) \leq 2n - t + r = 2n - \lfloor \frac{n}{2} \rfloor + 3k \leq 
	3 \cdot \lfloor \frac{n}{2} \rfloor + 2 + 3k$.  
	So to prove \eqref{eq:f(n-1,m-d)}, it suffices to show that $3k+1 < 3a$. 
	As $k \leq \sqrt{\frac{a}{3}}+1$, it suffices to show that $3\left( \sqrt{\frac{a}{3}}+1 \right) < 3a-1$. Rearranging and squaring, we get the inequality $9a^2-27a+16 > 0$, which holds for all $a \geq 3$. For $a=2$ we simply note that $k = 1$ and so $3k+1 = 4 < 6 = 3a$.

	Now we show that $g_2(n-1,m-d) \geq g(n,m) - 2$. 
	Fix $t,r$ that achieve the minimum in the definition of $g_2(n,m)$; so $t(n-t)+t_2(r) \geq m$ and 
	\begin{equation}\label{eq:g_2(n,m)}
	g_2(n,m) = 2n-t+r = 3t - 2 \cdot \left(2t - n - \frac{r}{2}\right).
	\end{equation} 
	By Lemma \ref{lem:range of t,r}, we may assume that $-\frac{1}{2} \leq 2t-n-\frac{r}{2} \leq \frac{1}{2}.$
%	\begin{equation}\label{eq:range of t,r}
%	-\frac{1}{2} \leq 2t-n-\frac{r}{2} \leq \frac{1}{2}.
%	\end{equation}
 	Fix also $t',r'$ such that $t'(n-1-t') + t_2(r') \geq m-d$ and $g_2(n-1,m-d) = 2(n-1) - t' + r'$. Suppose by contradiction that $g_2(n-1,m-d) \leq g_2(n,m)-3$. Then $2(n-1) - t' + r' \leq 2n - t + r - 3$, so $t' - t \geq r' - r + 1$. Put $c := t'-t$, so that $r' \leq r+c-1$. Then $t_2(r') \leq t_2(r+c-1)$. 
 	So we can write
	\begin{equation}\label{eq:m-d}
	\begin{split}
	m-d &\leq t'(n-1-t') + t_2(r') \leq (t+c)(n-t-c-1) + t_2(r+c-1) \\ &= t(n-t) + cn - (2c+1)t - (c+1)c + t_2(r+c-1).
	\end{split}
	\end{equation}
	Since $(t,r)$ achieves the minimum in the definition of $g_2(n,m)$, we must have $m-1 \geq t(n-t) + t_2(r-1)$, so $t(n-t) \leq m-1-t_2(r-1)$. Plugging this into \eqref{eq:m-d} and rearranging, we get 
	$$d \geq t + c(2t - n) + (c+1)c + 1 - t_2(r+c-1) + t_2(r-1).$$ 
	Note that $$t_2(r+c-1) - t_2(r-1) \leq \frac{(r+c-1)^2}{4} - \frac{(r-1)^2-1}{4} = \frac{cr}{2} + \frac{(c-1)^2}{4}.$$ So we get 
	\begin{equation}\label{eq:d-bound 1}
	d \geq t + c\left(2t-n-\frac{r}{2}\right) + (c+1)c + 1 - \frac{(c-1)^2}{4}.
	\end{equation}
%	We now derive another estimate similar to \eqref{eq:d-bound 1}, but this time we use the pair $(t+1,r)$ instead of the pair $(t,r-1)$. Since $(t,r)$ achieves the minimum in the definition of $g_2(n,m)$, we must have $m-1 \geq (t+1)(n-t-1) + t_2(r) = t(n-t) + n - 2t - 1 + t_2(r)$, so $t(n-t) \leq m - n + 2t - t_2(r)$. Plugging this into \eqref{eq:m-d} and rearranging, we get $d \geq t + (c-1)(2t-n) + (c+1)c - t_2(r+c-1) + t_2(r)$. 
%	Let $j \equiv r+c-1 \pmod{2}$, $j = 0,1$.
%	Note that $t_2(r+c-1) - t_2(r) = \frac{(r+c-1)^2-j}{4} - \frac{r^2-i}{4} = \frac{(c-1)r}{2} + \frac{(c-1)^2+i-j}{4}$. So we get
%	\begin{equation}\label{eq:d-bound 2}
%	d \geq t + (c-1)\left(2t-n-\frac{r}{2}\right) + (c+1)c - \frac{(c-1)^2+i-j}{4}.
%	\end{equation}
	
	We now complete the proof by considering the three possible values of $2t-n-\frac{r}{2}$. Suppose first that $2t-n-\frac{r}{2} = 0$.
	Then $g_2(n,m) = 3t$ by \eqref{eq:g_2(n,m)}. 
	By \eqref{eq:d-bound 1}, we have $d \geq t + (c+1)c + 1 - \frac{(c-1)^2}{4} = t + \frac{3}{4}(c+1)^2 \geq t = g_2(n,m)/3$, in contradiction to our assumption on $d$. 
	
	Suppose now that $2t-n - \frac{r}{2} = \frac{1}{2}$. Then $g_2(n,m) = 3t-1$ by \eqref{eq:g_2(n,m)}. By \eqref{eq:d-bound 1}, we have $d \geq t + \frac{c}{2} + (c+\nolinebreak 1)c + 1 - \frac{c^2-2c+1}{4} = t+\frac{3c^2}{4} + 2c+\frac{3}{4} > t-1$, where the last inequality holds for every $c$. So $d \geq t$ and hence $3d \geq 3t \geq g_2(n,m)$, a contradiction. 
	
	Finally, suppose that $2t-n - \frac{r}{2} = -\frac{1}{2}$. Then $g_2(n,m) = 3t+1$ by \eqref{eq:g_2(n,m)}. By \eqref{eq:d-bound 1}, we have $d \geq t - \frac{c}{2} + (c+\nolinebreak 1)c + 1 - \frac{c^2-2c+1}{4} = t+\frac{3c^2}{4} + c + \frac{3}{4} > t$, where the last inequality holds for every $c$. So $d \geq t+1$ and hence $3d \geq 3t+3 \geq g_2(n,m)$, a contradiction.  
\end{proof}

\begin{proof}[Proof of Theorem \ref{thm:k=2 main}]
	The proof is by induction on $n$. The base cases $n=1,2$ are trivial because for these $n$ there is no graph on $n$ vertices with $t_2(n)+1$ edges. 
	The case $n = 3$ is also easy to verify.
	So from now on let $n \geq 4$. 
	Let $x,y,z$ be a triangle in $G$. 
	% Let $d = \min\{d(x),d(y),d(z)\} - \frac{n}{2}$. 
	We consider several cases. After dealing with each case, we will assume in all subsequent cases that this case does not hold. 
	\paragraph{Case 1:} $d(x) + d(y) + d(z) \geq g_2(n,m)$. In this case, take $H$ to be the graph consisting of all edges touching $x,y,z$. This graph is chordal and clearly contains the edges of the triangle $x,y,z$. Also, $e(H) = d(x) + d(y) + d(z) - 3 \geq g_2(n,m) - 3$, as required. 
	\paragraph{Case 2:} There are distinct $u,v \in \{x,y,z\}$ such that $d(u) + d(v) \leq n$. Without loss of generality, suppose that $u = x, v = y$. Let $G' = G - \{x,y\}$. Then $
	e(G') = e(G) - d(x) - d(y) + 1 \geq m - n + 1$.
	By the induction hypothesis (applied to any arbitrary triangle in $G'$), $G'$ contains a chordal subgraph $H'$ with $e(H') \geq g_2(n-2,m-n+1) - 3 \geq g_2(n,m) - 6$, by Lemma \ref{lem:n-2}. Let $H := H' + \{xy,xz,yz\}$. Then $H$ is chordal, contains the edges of the triangle $x,y,z$, and satisfies $e(H) = e(H') + 3 \geq g_2(n,m) - 3$.
	
	\vspace{0.2cm}
	We claim that if cases 1-2 do not hold then $m \geq t_2(n)+2$. Indeed, suppose by contradiction that $m = t_2(n)+1$. We have $g_2(n,t_2(n)+1) = 2n - \lceil \frac{n}{2} \rceil + 2$. Since case 1 does not hold,
	$d(x) + d(y) + d(z) \leq g_2(n,t_2(n)+1) - 1 \leq \frac{3n}{2} + 1$. But then there are distinct $u,v \in \{x,y,z\}$ with $d(u) + d(v) \leq \frac{2}{3} \cdot \left( \frac{3n}{2} + 1 \right) \leq n + \frac{2}{3}$. Hence $d(u) + d(v) \leq n$, contradicting that case 2 does not hold. So $m \geq t_2(n)+2$. 
	
	\paragraph{Case 3:}
	There are distinct $u,v \in \{x,y,z\}$ such that the edge $uv$ is on exactly one triangle (namely the triangle $x,y,z$). Without loss of generality, suppose that $u = x, v = y$. Since case 2 does not hold, we have $d(x) + d(y) \geq n+1$. On the other hand, since $xy$ is on exactly one triangle, it must be that $d(x) + d(y) = n+1$, $N(x) \cap N(y) = \{z\}$ and every vertex in $V(G) \setminus \{x,y,z\}$ is adjacent to exactly one of $x,y$. 
	Let $G' = G - \{x,y\}$. Then 
	$e(G') = e(G) - d(x) - d(y) + 1 = m - n$.
    By Lemma~\ref{lem:n-2}, $e(G') \geq t_2(v(G')) + 1$, which means that $G'$ contains a triangle. 
	By the induction hypothesis (applied to an arbitrary triangle in $G'$), $G'$ contains a chordal subgraph $H'$ with 
	$
	e(H') \geq g_2(n-2,m-n) - 3 \geq g_2(n,m) - 7,
	$
	by Lemma \ref{lem:n-2}. 
	So it is enough to show that $G$ contains a chordal graph $H$ which contains the edges of the triangle $x,y,z$ and satisfies $e(H) = e(H') + 4$. Suppose first that $z$ is isolated in $H'$. We claim that there exists $w \in N_G(z) \setminus \{x,y\}$. Indeed, if not, then $d_G(z) = 2$. Also, as $d(x) + d(y) = n+1$, we have $d(x) \leq \lfloor \frac{n+1}{2} \rfloor$ or $d(y) \leq \lfloor \frac{n+1}{2} \rfloor$. Suppose this holds for $y$. Then $d(y) + d(z) \leq \lfloor \frac{n+1}{2} \rfloor + 2 \leq n$ for $n \geq 4$, so case 2 holds, contradiction. This proves our claim that there exists $w \in N_G(z) \setminus \{x,y\}$. 
	Now take $H = H' + \{zw,xy,xz,yz\}$. By adding the vertices in the order $z,x,y$, we always add a simplicial vertex. Hence, $H$ is chordal.
	
	Suppose now that $z$ is not isolated in $H'$, and let $w \in V(G) \setminus \{x,y,z\}$ such that $zw \in E(H')$. As mentioned above, $w$ is adjacent in $G$ to either $x$ or $y$; without loss of generality it is adjacent to $x$. Take $H = H' + \{xw,xy,xz,yz\}$. By adding to $H'$ first $x$ and then $y$, we always add a simplicial vertex. Hence, $H$ is chordal. 
%	Overall we showed that there exists a chordal subgraph $H$ of $G$ which contains the edges of the triangle $x,y,z$ and satisfies $e(H) = e(H') + 4 \geq f(n,m)-3$. 
	\paragraph{Case 4:}
	Cases 1-3 do not hold. Since case 1 does not hold, we have $d(x) + d(y) + d(z) \leq g_2(n,m) - 1$. Assume that $d := d(x) \leq d(y) \leq d(z)$; then $3d \leq g_2(n,m) - 1$.  
	Since case 3 does not hold, there exists $w \in V(G) \setminus \{x\}$ which is a common neighbour of $y,z$. 
	Set $G' = G - \{x\}$. 
	We have
	$e(G') = m - d$.
	By the induction hypothesis, $G'$ has a chordal subgraph $H'$ which contains the edges of the triangle $y,z,w$ and satisfies 
	$
	e(H') \geq g_2(n-1,m-d) - 3 \geq g_2(n,m) - 5,
	$
	by Lemma \ref{lem:f(n,m) main}.
	Let $H = H' + \{xy,xz\}$. Then $H$ is chordal, contains the edges of the triangle $x,y,z$, and satisfies $e(H) = e(H') + 2 \geq g_2(n,m) - 3$, as required. 
\end{proof}

\section{Proof of Theorem \ref{thm:k=3}}\label{sec:k=3}
For convenience, put 
$g_3(n) := 3n-\lceil\frac{n}{3} \rceil+2$. We prove Theorem \ref{thm:k=3} in the following stronger form. 
\begin{theorem}\label{thm:K4}
	Let $G$ be a graph with $n$ vertices and $t_3(n)+1$ edges, and let $X = \{x_1,x_2,x_3,x_4\}$ be a $4$-clique in $G$. Then $G$ has a chordal subgraph with at least $g_3(n) - 6$ edges which contains the edges of the clique $X$. 
\end{theorem}

\noindent
We need some simple facts on the numbers $t_3(n)$ and $g_3(n)$.
\begin{lemma}\label{lem:t3}
    For every $n \geq 5$, it holds that $t_3(n) - t_3(n-1) = \lfloor \frac{2n}{3} \rfloor$, 
    $t_3(n) - t_3(n-2) = \lfloor \frac{4n}{3} \rfloor - 1$, 
    $t_3(n) - t_3(n-3) = 2n-3$, $t_3(n) - t_3(n-4) = g_3(n) - 7$.
\end{lemma}
\begin{proof}
For $i = 1,2,3$, $T_3(n-i)$ is obtained from $T_3(n)$ by deleting one vertex from each of the $i$ largest classes of $T_3(n)$. For $i = 1$, the deleted vertex has degree $\lfloor \frac{2n}{3} \rfloor$. For $i = 2$, the sum of degrees of the deleted vertices is $\lfloor \frac{4n}{3} \rfloor$, and these vertices are adjacent. For $i = 3$, the sum of degrees of the deleted vertices is $2n$, and they form a triangle. Finally, $t_3(n) - t_3(n-4) = t_3(n) - t_3(n-1) + t_3(n-1) - t_3(n-4) = \lfloor \frac{2n}{3} \rfloor + 2(n-1)-3 = 3n-\lceil \frac{n}{3} \rceil - 5 = g_3(n) - 7$.
\end{proof}
% \begin{lemma}\label{lem:f_3}
%     For every $n \geq 5$, it holds that $f(n) - f(n-1) = 3$ if $n \equiv 0,2 \pmod{3}$ and $f(n) - f(n-1) = 2$ if $n \equiv 1 \pmod {3}$. Hence, $f(n) - f(n-2) = 5$ if $n \equiv 1,2 \pmod{3}$ and $f(n) - f(n-2) = 6$ if $n \equiv 0 \pmod {3}$. Also, $f(n) - f(n-3) = 8$ and $f(n) - f(n-4) \leq 11$.
% \end{lemma}
\begin{lemma}\label{lem:f_3}
    For every $n \geq 5$, it holds that $g_3(n) - g_3(n-1) = 3$ if $n \equiv 0,2 \pmod{3}$ and $g_3(n) - g_3(n-1) = 2$ if $n \equiv 1 \pmod {3}$. Hence, $g_3(n) - g_3(n-2) \leq 6$, $g_3(n) - g_3(n-3) = 8$ and $g_3(n) - g_3(n-4) \leq 11$.
\end{lemma}
\begin{proof}
$g_3(n) - g_3(n-1) = 3 - \lceil \frac{n}{3} \rceil + \lceil \frac{n-1}{3} \rceil$, and it is easy to see that $\lceil \frac{n}{3} \rceil - \lceil \frac{n-1}{3} \rceil$ equals $0$ if $n \equiv 0,2 \pmod{3}$ and equals $1$ if $n \equiv 1 \pmod{3}$. 
\end{proof}
% \begin{lemma}\label{lem:n-4}
% 	For every $n \geq 5$, $f(n) - 7 = t_3(n) - t_3(n-4)$. 
% \end{lemma}

We also need the following simple lemma saying that in a graph with $t_3(n)+1$ edges, we can find two $4$-cliques sharing 3 vertices. This fact is originally due to Dirac \cite{Dirac}. For completeness, we include a proof.
\begin{lemma}\label{lem:diamond}
A graph $G$ with $n \geq 5$ vertices and $t_3(n) + 1$ edges has two $4$-cliques sharing 3 vertices. 
\end{lemma}
\begin{proof}
The proof is by induction on $n$. For the base case $n = 5$, take a $4$-clique and notice that the remaining vertex must send at least 3 edges to this $4$-clique as $t_3(5)+1 = 9$. Let $n \geq 6$. Take a $4$-clique $x_1,\dots,x_4$. If there is a vertex outside of $x_1,\dots,x_4$ which has three neighbours in $x_1,\dots,x_4$ then we are done. Else, $d(x_1) + \dots + d(x_4) \leq 2n+4$, so there is $i \in [4]$ such that $d(x_i) \leq \lfloor \frac{n}{2}\rfloor + 1 \leq \lfloor \frac{2n}{3}\rfloor$, where the last inequality holds for all $n \geq 6$. So $G' := G-x_i$ has at least $t_3(n-1)+1$ edges, and we can apply induction.
\end{proof}

\begin{proof}[Proof of Theorem \ref{thm:K4}]
	The proof is by induction on $n$. The claim is trivial if $n \leq 4$ so suppose that $n \geq 5$. We proceed by a sequence of claims. 
	\begin{claim}\label{claim:degenerate case 0}
		We may assume that there exists $y \in V(G) \setminus \{x_1,\dots,x_4\}$ which is adjacent to at least 3 of the vertices $x_1,\dots,x_4$.
	\end{claim}
	\begin{proof}[Proof of Claim \ref{claim:degenerate case 0}]
		Suppose that there is no such $y$. We will show that the assertion of Theorem \ref{thm:K4} holds. 
		For $0 \leq j \leq 2$, let $A_j$ be the set of vertices in $V(G) \setminus \{x_1,\dots,x_4\}$ which are adjacent to $j$ of the vertices $x_1,\dots,x_4$, and let $a_j = |A_j|$. Then $a_0 + a_1 + a_2 = n-4$ and $a_1 + 2a_2 = d(x_1) + \dots + d(x_4) - 12$. So 
		\begin{equation}\label{eq:degenerate case}
		d(x_1) + \dots + d(x_4) = n + 8 + a_2 - a_0.
		\end{equation}
		We consider three cases:
		\paragraph{Case 1:} $0 < a_2 \leq n-5$. Then $d(x_1) + \dots + d(x_4) \leq 2n+3$ by \eqref{eq:degenerate case}. 
		Let $G'$ be the graph obtained from $G$ by deleting $x_2,x_3,x_4$ and all edges touching $x_1$. Then $v(G') = n-3$ and 
		$e(G') = e(G) - (d(x_1) + \dots + d(x_4)) + 6 \geq e(G) - 2n + 3 = t_3(n) + 1 - 2n + 3 = t_3(n-3) + 1.$
%		\begin{equation}\label{eq:degenerate case G'}
%		e(G') = e(G) - (d(x_1) + \dots + d(x_4)) + 6 \geq e(G) - 2n + 3 = t_3(n) + 1 - 2n + 3 = t_3(n-3) + 1.
%		\end{equation}
        In particular, $n-3 \geq 4$ (else $G'$ cannot have more than $t_3(n-3)$ edges).
		By the induction hypothesis (applied to an arbitrary $4$-clique in $G'$), $G'$ contains a chordal subgraph $H'$ with $e(H') \geq g_3(n-3)-6$. Take $z \in A_2$, and assume without loss of generality that $z$ is adjacent to $x_1,x_2$. Let $H = H' + \{x_1z,x_2z\} + \{x_ix_j : 1 \leq i < j \leq 4\}$. Then $e(H) = e(H') + 8 \geq g_3(n-3) + 8 - 6 = g_3(n) - 6$. Also, $H$ is chordal: adding $x_1z$ to $H'$ keeps it chordal because $x_1$ is isolated in $H'$, and then by adding $x_2,x_3,x_4$ in this order we always add a simplicial vertex. 
		\paragraph{Case 2:} $a_2 = n-4$. Then every vertex outside of $\{x_1,\dots,x_4\}$ is adjacent to two of the vertices $x_1,\dots,x_4$. By \eqref{eq:degenerate case}, $d(x_1) + \dots + d(x_4) = 2n+4$. Pick an arbitrary $v \in V(G) \setminus \{x_1,\dots,x_4\}$ and suppose without loss of generality that $v$ is adjacent to $x_1$. Let $G'$ be the graph obtained from $G$ by deleting $x_2,x_3,x_4$ and all edges touching $x_1$ except for $x_1v$. Then $v(G') = n-3$ and $e(G') = e(G) - (d(x_1) + \dots + d(x_4)) + 7 = e(G) - 2n + 3 = t_3(n-3) + 1$. 
		Pick an arbitrary $4$-clique $y_1,\dots,y_4$ in $G'$. By the induction hypothesis, $G'$ contains a chordal subgraph $H'$ with $e(H') \geq g_3(n-3) - 6$ such that $H'$ contains the edges of the clique $y_1,\dots,y_4$. Each $y_i$ has two neighbours in $x_1,\dots,x_4$. By pigeonhole, two of the $y_i$'s have a common neighbour. Without loss of generality, we may assume that $y_1,y_2$ are both adjacent to $x_k$ for some $k \in [4]$. Suppose that $y_1$ is also adjacent to $x_{\ell}$. Let $H := H' - x_1v + \{x_ky_1,x_ky_2,x_{\ell}y_1\} +  \{x_ix_j : 1 \leq i < j \leq 4\}$, see Figure 2(a). Then $e(H) \geq e(H') + 8$ (since we add $9$ edges and delete at most $1$), so $e(H) \geq g_3(n-3) + 8 - 6 = g_3(n) - 6$. Also, $H' - x_1v$ is chordal because $x_1$ is a leaf or an isolated vertex in $H'$. 
		Now observe that $H$ is chordal: adding first $x_k$, then $x_{\ell}$ and then the other two vertices among $x_1,\dots,x_4$, we add a simplicial vertex at each step. 
		\paragraph{Case 3:} $a_2 = 0$. 
		Then there are at most $n-4$ edges between $x_1,\dots,x_4$ and $V(G) \setminus \{x_1,\dots,x_4\}$.
		Also, $d(x_1) + \dots + d(x_4) \leq n+8$ by \eqref{eq:degenerate case}. Let $G'$ be the graph obtained from $G$ by deleting $x_2,x_3,x_4$ and all edges touching $x_1$. Then $v(G') = n-3$ and $e(G') \geq e(G) - (n + 2) = t_3(n) + 1 - (n + 2) = t_3(n-3) + 1 + n - 5$. Note that $n \geq 8$ because else $e(G) \leq \binom{n-4}{2} + n - 4 + 6 < t_3(n)$ (where the last inequality holds for $n \leq 7$), a contradiction. So $v(G') \geq 5$.
		By Lemma \ref{lem:diamond}, there are distinct $v_1,v_2,w_1,w_2,w_3 \in V(G')$ such that $v_i,w_1,w_2,w_3$ is a $4$-clique for $i = 1,2$. Let $G''$ be the graph obtained from $G'$ by deleting all edges touching $v_1$. Then $v(G'') = n-3$ and $e(G'') \geq e(G') - (n-5)$ because $v_1$ is not adjacent in $G'$ to $x_1$ or to itself, leaving $n-5$ potential neighbours. So $e(G'') \geq t_3(n-3) + 1$. By the induction hypothesis, $G''$ contains a chordal subgraph $H'$ with $e(H') \geq g_3(n-3) - 6$, such that $H'$ contains the edges of the clique $v_2,w_1,w_2,w_3$. Let $H = H' + \{v_1w_i : 1 \leq i \leq 3\} + \{x_ix_j : 1 \leq i < j \leq 4\}$. Then $e(H) = e(H') + 9 \geq g_3(n) - 6$ and it is easy to check that $H$ is chordal.  
	\end{proof}
	\begin{claim}\label{claim:degenerate case}
		We may assume that for every $i \in [4]$, the following holds. 
		If $d(x_i) \geq \lfloor \frac{2n}{3} \rfloor + 1$, then there is $y \in V(G) \setminus \{x_1,\dots,x_4\}$ such that $y$ is adjacent to $x_i$ and to at least two vertices in $X \setminus \{x_i\}$. 
	\end{claim}
	\begin{proof}[Proof of Claim \ref{claim:degenerate case}]
		Without loss of generality, $i = 1$. Suppose that $d(x_1) \geq \lfloor \frac{2n}{3} \rfloor + 1$ but there is no $y$ as above. We will show that the assertion of Theorem \ref{thm:K4} holds. 
		For $0 \leq j \leq 3$, let $A_j$ be the set of vertices $z \in V(G) \setminus \{x_2,x_3,x_4\}$ which are adjacent to exactly $j$ of the vertices $x_2,x_3,x_4$, and let $a_j = |A_j|$. Note that $x_1 \in A_3$. 
		By Claim \ref{claim:degenerate case 0}, we may assume that there exists $z \in V(G) \setminus \{x_1,\dots,x_4\}$ which is adjacent to three of the vertices $x_1,\dots,x_4$. By assumption, these three vertices cannot include $x_1$, so we must have $z \in A_3$. Hence, 
		$A_3 \setminus \{x_1\} \neq \emptyset$.  
		
		We have $a_0 + a_1 + a_2 + a_3 = n-3$ and $a_1 + 2a_2 + 3a_3 = d(x_2) + d(x_3) + d(x_4) - 6$. So $a_2 + 2a_3 = d(x_2) + d(x_3) + d(x_4) - n - 3 + a_0$. By assumption, there is no $y \in V(G) \setminus \{x_1,\dots,x_4\}$ which is adjacent to $x_1$ and belongs to $A_2 \cup A_3$. Hence, $d(x_1) \leq n - a_2 - a_3 = 2n - d(x_2) - d(x_3) - d(x_4) + 3 + a_3 - a_0$. So we get
		\begin{equation}\label{eq:degenerate x1,...,x4}
		d(x_1) + \dots + d(x_4) \leq 2n+3 + a_3 - a_0.
		\end{equation}
		\noindent
		Pick a set of edges $F$ as follows: 
		\begin{itemize}
			\item Suppose first that $a_0 = 0$. We claim that $A_1 \cup A_2 \neq \emptyset$. Indeed, if $A_1 \cup A_2 = \emptyset$ then $A_3 = V(G) \setminus \nolinebreak \{x_2,x_3,x_4\}$. But then $N(x_1) = \{x_2,x_3,x_4\}$ so $d(x_1) = 3$. On the other hand, $d(x_1) \geq \lfloor \frac{2n}{3} \rfloor + 1$ by assumption. So $\lfloor \frac{2n}{3} \rfloor \leq 2$, which is false for every $n \geq 5$. Now pick an arbitrary $v \in A_1 \cup A_2$, and suppose without loss of generality that $v$ is adjacent to $x_4$. Take $F$ to be the set of all edges 
			between $x_4$ and $N(x_4) \setminus (A_3 \cup \{x_2,x_3,v\})$. 
			% touching $x_4$ except for the edges between $x_4$ and $A_3$ and the edge $x_4v$. 
			Note that $|F| = d(x_4) - a_3 - 3$.
			\item If $a_0 \geq 1$ then $F$ is the set of all edges 
			% touching $x_4$ except for the edges between $x_4$ and $A_3$. 
			between $x_4$ and $N(x_4) \setminus (A_3 \cup \{x_2,x_3\})$. 
			Then $|F| = d(x_4) - a_3 - 2$. 
		\end{itemize}
		In both cases we have $|F| \leq d(x_4) - a_3 - 3 + a_0$.   
		Let $G'$ be the graph obtained from $G$ by deleting $x_1,x_2,x_3$ and all edges in $F$. So $v(G') = n-3$. To count the deleted edges, note that there are $d(x_1) + d(x_2) + d(x_3) - 3$ edges touching $x_1,x_2,x_3$, and that the edges in $F$ do not touch $x_1,x_2,x_3$. So the number of deleted edges is 
		$$
		d(x_1) + d(x_2) + d(x_3) - 3 + |F| \leq d(x_1) + \dots + d(x_4) - a_3 + a_0 - 6 \leq 2n-3,
		$$
		where the last inequality uses \eqref{eq:degenerate x1,...,x4}. So $e(G') \geq e(G) - 2n + 3 = t_3(n) + 1 - 2n + 3 = t_3(n-3) + 1$. By the induction hypothesis, $G'$ contains a chordal subgraph $H'$ with $e(H') \geq g_3(n-3) - 6$. 
		Define a subgraph $H''$ of $G'$ as follows: 
		\begin{itemize}
			\item If there is $z \in A_3 \setminus \{x_1\}$ such that $x_4z \in E(H')$, then $H'' = H'$;
			\item Else, pick any $z \in A_3 \setminus \{x_1\}$ and set $H'' = H' - x_4v + x_4z$.
			% , where $v$ is as in the first item in the definition of $F$.
		\end{itemize}
		We claim that $e(H'') \geq e(H')$. In the first item this is obvious. Suppose we are in the second item. By the definition of $F$, every neighbour of $x_4$ in $G'$ belongs to $\{v\} \cup A_3 \setminus \{x_1\}$. Since we are in the second item, the only possible neighbour of $x_4$ in $H'$ is $v$. So indeed $e(H'') \geq e(H')$. Also, $H''$ is chordal; in the second case, this is because $x_4$ is a leaf in both $H'$ and $H''$. 
		
		By the definition of $H''$, there is $z \in A_3 \setminus \{x_1\}$ such that $x_4z \in E(H'')$. Now put $H = H'' + \{x_2z,x_3z\} + \{x_ix_j : 1 \leq i < j \leq 4\}$, see Figure 2(b). Then $e(H) = e(H'') + 8 \geq g_3(n-3) + 8 - 6 = g_3(n) - 6$. Also, $H$ is chordal: adding the vertices $x_2,x_3,x_1$ in this order, we always add a simplicial vertex. This \nolinebreak proves \nolinebreak Claim \nolinebreak \ref{claim:degenerate case}. 
	\end{proof}
	\begin{figure}[h]
\centering 
\begin{subfigure}[b]{0.3\textwidth}
\begin{tikzpicture}[scale = 1.25]
%	\coordinate (y1) at (0,0);
%	\coordinate (y2) at (1,0);
%	\coordinate (y3) at (0,1);
%	\coordinate (y4) at (1,1);
%	\coordinate (x1) at (0,-1);
%	\coordinate (x2) at (1,-1);
%	\coordinate (x3) at (0,-2);
%	\coordinate (x4) at (1,-2);
%	
%	\foreach \i in {1,2,3,4}
%	{
%		\draw (x\i) node[fill=black,circle,minimum size=2pt,inner sep=0pt] {};
%		\draw (y\i) node[fill=black,circle,minimum size=2pt,inner sep=0pt] {};
%	}
%
%	\draw (y1) node[left] {$y_1$};
%	\draw (y2) node[right] {$y_2$};
%	\draw (y3) node[left] {$y_3$};
%	\draw (y4) node[right] {$y_4$};
%	
%	\draw (x1) node[left] {$x_1$};
%	\draw (x2) node[right] {$x_2$};
%	\draw (x3) node[left] {$x_3$};
%	\draw (x4) node[right] {$x_4$};

	\coordinate (y1) at (0,0);
	\coordinate (y2) at (0,1);
	\coordinate (y3) at (1,0);
	\coordinate (y4) at (1,1);
	%	\coordinate (x1) at (-0.7,-0.7);
	%	\coordinate (x2) at (-0.7,0.3);
	%	\coordinate (x3) at (-1.7,0.3);
	%	\coordinate (x4) at (-1.7,-0.7);
	\coordinate (x1) at (-1,0);
	\coordinate (x2) at (-1,-1);
	\coordinate (x3) at (-2,0);
	\coordinate (x4) at (-2,-1);
	\coordinate (v) at (-1.5,0.5);
	\foreach \i in {1,2,3,4}
	{
		\draw (x\i) node[fill=black,circle,minimum size=2pt,inner sep=0pt] {};
		\draw (y\i) node[fill=black,circle,minimum size=2pt,inner sep=0pt] {};
	}

    \draw (v) node[fill=black,circle,minimum size=2pt,inner sep=0pt] {};
	
	\draw (x1) node[above] {$x_1$};
	\draw (x4) node[below] {$x_4$};
	\draw (x2) node[below] {$x_2$};
	\draw (x3) node[above] {$x_3$};
	\draw (y1) node[below] {$y_1$};
	\draw (y2) node[above] {$y_2$};
	\draw (y3) node[below] {$y_3$};
	\draw (y4) node[above] {$y_4$};
	\draw (v) node[above] {$v$};
	
	\draw (y1) -- (y2) -- (y3) -- (y4) -- (y1); \draw (y2) -- (y4); \draw (y1) -- (y3);
% 	\draw[dashed] (x1) -- (x2) -- (x3) -- (x4) -- (x1); \draw[dashed] (x2) -- (x4); \draw[dashed] (x1) -- (x3);
% 	\draw[dashed] (x1) -- (y1);
% 	\draw[dashed] (x1) -- (y2);
% 	\draw[dashed] (x2) -- (y1);
    \draw[color=red] (x1) -- (x2) -- (x3) -- (x4) -- (x1); \draw[color=red] (x2) -- (x4); \draw[color=red] (x1) -- (x3);
	\draw[color=red] (x1) -- (y1);
	\draw[color=red] (x1) -- (y2);
	\draw[color=red] (x2) -- (y1);
	\draw[dashed] (x1) -- (v);
\end{tikzpicture}
\caption{\centering Case 2 in the proof of Claim \ref{claim:degenerate case 0} ($k = 1,\ell=2$).}
\end{subfigure}
% \hspace{1cm} 
	\begin{subfigure}[b]{0.3\textwidth} 
	\hspace{0.2cm} 
	\begin{tikzpicture}[scale=1.5]
	\coordinate (x4) at (0,0);
	\coordinate (z) at (0,1);
	\coordinate (x2) at (0.5,-0.5);
	\coordinate (x3) at (-0.5,-0.5);
	\coordinate (x1) at (0,-1);
	\coordinate (v) at (-1,0.5);
	
	\foreach \i in {1,2,3,4}
	{
		\draw (x\i) node[fill=black,circle,minimum size=2pt,inner sep=0pt] {};
	}

	\draw (z) node[fill=black,circle,minimum size=2pt,inner sep=0pt] {};
	\draw (v) node[fill=black,circle,minimum size=2pt,inner sep=0pt] {};
	
	\draw (x1) node[below] {$x_1$};
	\draw (x2) node[right] {$x_2$};
	\draw (x3) node[left] {$x_3$};
	\draw (x4) node[right] {$x_4$};
	\draw (z) node[above] {$z$};
	\draw (v) node[left] {$v$};
	
	\draw[color=red] (x1) -- (x2) -- (x3) -- (x4) -- (x1); \draw[color=red] (x2) -- (x4); \draw[color=red] (x1) -- (x3);
	\draw (x4) -- (z);
	\draw[color=red] (x3) -- (z);
	\draw[color=red] (x2) -- (z);
	\draw[dashed] (x4) -- (v);
	\end{tikzpicture}
	\caption{Proof of Claim \ref{claim:degenerate case}. 
	% The deleted edge $x_4v$ is dashed. %The added edges are dashed. 
	}
\end{subfigure} 
\begin{subfigure}[b]{0.3\textwidth} 
% \vspace{1cm}
\begin{tikzpicture}[scale = 1.25]
\coordinate (w1) at (0,0);
\coordinate (w2) at (0,1);
\coordinate (w3) at (1,0);
\coordinate (w4) at (1,1);
%	\coordinate (x1) at (-0.7,-0.7);
%	\coordinate (x2) at (-0.7,0.3);
%	\coordinate (x3) at (-1.7,0.3);
%	\coordinate (x4) at (-1.7,-0.7);
\coordinate (x1) at (-1,-1);
\coordinate (x2) at (-1,0);
\coordinate (x3) at (-2,0);
\coordinate (x4) at (-2,-1);
\foreach \i in {1,2,3,4}
{
	\draw (x\i) node[fill=black,circle,minimum size=2pt,inner sep=0pt] {};
	\draw (w\i) node[fill=black,circle,minimum size=2pt,inner sep=0pt] {};
}
\draw (w1) -- (w2) -- (w3) -- (w4) -- (w1); \draw (w1) -- (w3); \draw (w2) -- (w4);
% \draw[dashed] (x1) -- (x2) -- (x3) -- (x4) -- (x1); \draw[dashed] (x1) -- (x3); \draw[dashed] (x2) -- (x4);
% \draw[dashed] (x1) -- (w1); \draw[dashed] (x1) -- (w2); \draw[dashed] (x1) -- (w3);
% \draw[dashed] (x2) -- (w1); \draw[dashed] (x2) -- (w2);
% \draw[color=blue] (x1) -- (x2) -- (x3) -- (x4) -- (x1); \draw[color=blue] (x1) -- (x3); \draw[color=blue] (x2) -- (x4);
% \draw[color=blue] (x1) -- (w1); \draw[color=blue] (x1) -- (w2); \draw[color=blue] (x1) -- (w3);
% \draw[color=blue] (x2) -- (w1); \draw[color=blue] (x2) -- (w2);
\draw[color=red] (x1) -- (x2) -- (x3) -- (x4) -- (x1); \draw[color=red] (x1) -- (x3); \draw[color=red] (x2) -- (x4);
\draw[color=red] (x1) -- (w1); \draw[color=red] (x1) -- (w2); \draw[color=red] (x1) -- (w3);
\draw[color=red] (x2) -- (w1); \draw[color=red] (x2) -- (w2);

\draw (x1) node[below] {$x_1$};
\draw (x4) node[below] {$x_4$};
\draw (x2) node[above] {$x_2$};
\draw (x3) node[above] {$x_3$};
\draw (w1) node[below] {$w_1$};
\draw (w2) node[above] {$w_2$};
\draw (w3) node[below] {$w_3$};
\draw (w4) node[above] {$w_4$};
\end{tikzpicture}
\caption{\centering Proof of Claim \ref{claim:xi degree}.}
\end{subfigure} 
\caption{\centering Illustrations to proofs of claims. The red edges are added after applying induction. The dashed edges are deleted.}
\end{figure}
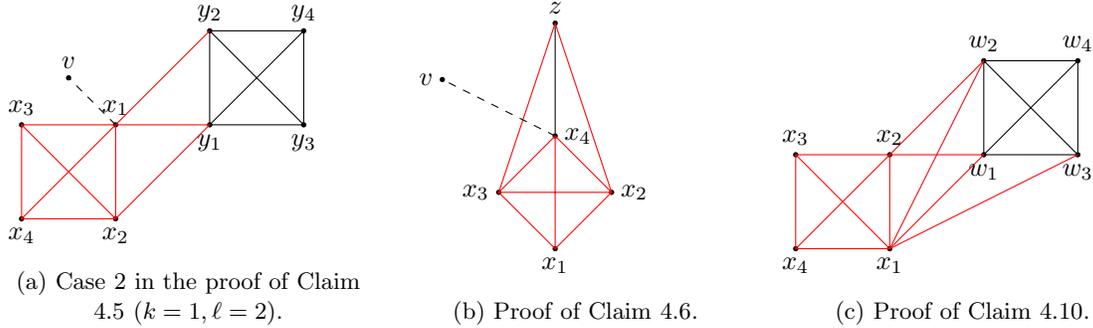

	We now continue with the proof of the theorem. 
	For $i \in [4]$, we say that $x_i$ is {\em deletable} if 
	$d(x_i) \leq \lfloor \frac{2n}{3} \rfloor$ and there is $y \in V(G) \setminus \{x_1,\dots,x_4\}$ which is adjacent to $x_j$ for all $j \in [4] \setminus \{i\}$. 
	\begin{claim}\label{claim:induction}
		We may assume that there is no deletable $x_i$. 
	\end{claim}
	\begin{proof}
		Suppose that $x_i$ is deletable. Without loss of generality, assume that $i = 1$. Let $G' = G - x_1$. Then 
		$e(G') = e(G) - d(x_1) \geq t_3(n) + 1 -\lfloor \frac{2n}{3} \rfloor = t_3(n-1) + 1$. Let $y$ be a common neighbour of $x_2,x_3,x_4$. By the induction hypothesis, $G'$ contains a chordal subgraph $H'$ with $e(H') \geq g_3(n-1) - 6$, such that $H'$ contains the edges of the clique $x_2,x_3,x_4,y$. Take $H = H' + \{x_1x_2,x_1x_3,x_1x_4\}$. Then $H$ is chordal and $e(H) = e(H') + 3 \geq g_3(n-1) + 3 - 6 \geq g_3(n) - 6$. 
	\end{proof}
	\begin{claim}\label{claim:clique star}
		We may assume that $d(x_1) + \dots + d(x_4) \leq g_3(n) - 1$.
	\end{claim}
	\begin{proof}
		Else, take $H$ to be the subgraph of $G$ consisting of all edges which touch $x_1,\dots,x_4$. Then $H$ is chordal and $e(H) = d(x_1) + \dots + d(x_4) - 6 \geq g_3(n) - 6$. 
	\end{proof}
	
	\begin{claim}\label{claim:K5}
		We may assume that there is no $5$-clique containing $x_1,\dots,x_4$.
	\end{claim}
	\begin{proof}
		Suppose that $x_1,\dots,x_4,y$ is a $5$-clique. 
		If $d(x_i) \leq \lfloor \frac{2n}{3} \rfloor$ then $x_i$ is deletable. %(because $y$ is a common neighbour of the other three $x_j$'s). 
		Assuming no $x_i$ is deletable, we have $d(x_1) + \dots + d(x_4) \geq 4 \cdot (\lfloor \frac{2n}{3} \rfloor + 1) \geq 2n + \lfloor \frac{2n}{3}\rfloor + 2= g_3(n)$, so we are done by \nolinebreak Claim \nolinebreak \ref{claim:clique star}.
	\end{proof}
	
	\begin{claim}\label{claim:xi degree}
        We may assume that for every $i \in [4]$, $\sum_{j \neq i}{d(x_j)} \leq 2n$.
    \end{claim}
    \begin{proof}
% 		The second part of the claim follows from the first part because $d(x_1),d(x_2) \geq \lfloor \frac{2n}{3} \rfloor + 1$. 
		Suppose that $d(x_1) + d(x_2) + d(x_3) \geq 2n+1$ (the proof for each other $i$ is symmetric). 
		By Claim \ref{claim:clique star}, we may assume that
		$d(x_4) \leq g_3(n) - 1 - (2n+1) = n - \lceil \frac{n}{3}\rceil =  \lfloor \frac{2n}{3} \rfloor$.
		By Claim \ref{claim:induction}, we may assume that $x_4$ is not deletable. Then $x_1,x_2,x_3$ have no common neighbour except for $x_4$. But as $d(x_1) + d(x_2) + d(x_3) \geq 2n+1$, this means that every vertex in $V(G) \setminus \{x_1,\dots,x_4\}$ is adjacent to exactly two of the vertices $x_1,x_2,x_3$. Let $G' = G - \{x_1,\dots,x_4\}$. Then $v(G') = n-4$ and 
		$e(G') = e(G) - (d(x_1) + \dots + d(x_4)) + 6 \geq e(G)  - g_3(n) + 7 = t_3(n) + 1 - g_3(n) + 7 = t_3(n-4) + 1$ by Lemma \ref{lem:t3}. Fix an arbitrary $4$-clique $w_1,\dots,w_4$ in $G'$. By the induction hypothesis, $G'$ contains a chordal subgraph $H'$ with $e(H') \geq g_3(n-4) - 6$ such that $H'$ contains the edges of the clique $w_1,\dots,w_4$. Each $w_i$ is adjacent to two of the vertices $x_1,x_2,x_3$. By the pigeonhole principle and by symmetry, we may assume that each of $w_1,w_2$ is adjacent to $x_1,x_2$. Moreover, $w_3$ must be adjacent to $x_1$ or $x_2$; say to $x_1$. Take $H = H' + \{x_1w_1,x_1w_2,x_1w_3,x_2w_1,x_2w_2\} + \{x_ix_j : 1 \leq i < j \leq 4\}$, see Figure 2(c). Then $e(H) = e(H') + 11 \geq g_3(n-4) + 11 - 6 \geq g_3(n) - 6$. Also, $H$ is chordal: adding the vertices $x_1,x_2,x_3,x_4$ in this order, we always add a simplicial vertex. 
		This proves Claim \ref{claim:xi degree}.
	\end{proof}
    % \begin{proof}
    %     Suppose that $d(x_1) + d(x_2) + d(x_3) \geq 2n+2$ (the proof for $x_4$ is symmetric). Then $x_1,x_2,x_3$ have a common neighbour different from $x_4$. Also, by Claim \ref{claim:clique star}, we may assume that $d(x_4) \leq g_3(n) - 1 - (2n+2) \leq \lfloor \frac{2n}{3} \rfloor$. Hence, $x_4$ is deletable. 
    % \end{proof}

	By Claim \ref{claim:degenerate case 0}, we may assume that there is $y_0 \in V(G) \setminus \{x_1,\dots,x_4\}$ which is adjacent to at least 3 of the vertices $x_1,\dots,x_4$; say to $x_2,x_3,x_4$. By Claim \ref{claim:induction}, we may assume that $x_i$ is not deletable for any $i \in [4]$. Since $x_1$ is not deletable, $d(x_1) \geq \lfloor \frac{2n}{3} \rfloor + 1$. By Claim \ref{claim:degenerate case}, there is $z_0 \in V(G) \setminus \{x_1,\dots,x_4\}$ which is adjacent to $x_1$ and to at least two of the vertices $x_2,x_3,x_4$; without loss of generality, $z_0$ is adjacent to $x_3,x_4$. Since $x_2$ is not deletable, we have $d(x_2) \geq \lfloor \frac{2n}{3} \rfloor + 1$. Also, we may assume that $y_0 \neq z_0$ because else we would have a $5$-clique containing $x_1,\dots,x_4$. 
	
	\begin{claim}\label{claim:x3,x4 degrees}
		We may assume that $d(x_i) \leq \lfloor \frac{2n}{3} \rfloor$ for each $i = 3,4$.
	\end{claim}
	\begin{proof}
	By Claim \ref{claim:xi degree}, we may assume that $d(x_i) \leq 2n - d(x_1) - d(x_2) \leq 2n - 2 \cdot \left( \lfloor \frac{2n}{3} \rfloor + 1 \right) \leq \lfloor \frac{2n}{3} \rfloor$.
	\end{proof}
	If $d(y_0) \leq \lfloor \frac{2n}{3} \rfloor$ then we can delete $y_0$, apply induction to the remaining graph to find a chordal subgraph $H'$ with at least $g_3(n-1)-6$ edges which contains the edges of the clique $x_1,\dots,x_4$, and then add the edges $x_2y_0,x_3y_0,x_4y_0$ to $H'$ to conclude the proof. So suppose that $d(y_0) \geq \lfloor \frac{2n}{3} \rfloor+1$. 
Let $v \notin \{x_3,x_4\}$ be a common neighbour of $x_2$ and $y_0$. Such a vertex $v$ exists because $d(x_2) + d(y_0) \geq 2 \cdot \left( \lfloor \frac{2n}{3} \rfloor + 1 \right) \geq n+3$, where the last inequality holds for every $n \geq 5$. Note that $v \neq x_1$ because $x_1$ is not adjacent to $y_0$, as otherwise $x_1,\dots,x_4,y_0$ would be a $5$-clique. Similarly, $v \neq z_0$ because $v$ is adjacent to $x_2$ and $z_0$ is adjacent to $x_1,x_3,x_4$ (so otherwise $x_1,\dots,x_4,z_0$ would be a $5$-clique).
	
	Suppose first that $x_2,y_0,v$ have no common neighbour. Then $d(x_2) + d(y_0) + d(v) \leq 2n$. Let $G' = G - \{x_2,y_0,v\}$. Then $e(G') = e(G) - d(x_2) - d(y_0) - d(v) + 3 \geq t_3(n) + 1 - 2n + 3 = t_3(n-3) + 1$. By the induction hypothesis, $G'$ contains a chordal subgraph $H'$ with $e(H') \geq g_3(n-3) - 6$ such that $H'$ contains the edges of the clique $x_1,x_3,x_4,z_0$. Take $H = H' + \{x_2x_1,x_2x_3,x_2x_4,y_0x_2,y_0x_3,y_0x_4,vx_2,vy_0\}$, see Figure 3(a). Then $e(H) = e(H') + 8 \geq g_3(n-3) + 8 - 6 = g_3(n) - 6$. Also, $H$ is chordal: adding the vertices $x_2,y_0,v$ in this order, we always add a simplicial vertex.  
	
	Suppose now that $x_2,y_0,v$ have a common neighbour $w$. First, suppose that $w \in \{x_3,x_4\}$, say $w = x_3$. Let $G' = G - \{x_1,x_4\}$. We have $d(x_1) + d(x_4) \leq 2n - d(x_2) \leq 2n - \lfloor \frac{2n}{3} \rfloor - 1 \leq \lfloor \frac{4n}{3} \rfloor$, where the first inequality is by Claim \ref{claim:xi degree}. So $e(G') = e(G) - d(x_1) - d(x_4) + 1 \geq t_3(n) + 1 - \lfloor \frac{4n}{3} \rfloor + 1 = t_3(n-2) + 1$. By the induction hypothesis, $G'$ has a chordal subgraph $H'$ with $e(H') \geq g_3(n-2) - 6$ such that $H'$ contains the edges of the clique $x_2,y_0,v,x_3$. Take $H = H' + \{x_4x_2,x_4x_3,x_4y_0,x_1x_2,x_1x_3,x_1x_4\}$, see Figure 3(b). Then $e(H) = e(H') + 6 \geq g_3(n-2) + 6 - 6 \geq g_3(n) - 6$. Also, $H$ is chordal: adding $x_4$ and then $x_1$, we always add a simplicial vertex.
	Now suppose that $w \notin \{x_3,x_4\}$. Let $G' = G - \{x_1,x_3,x_4\}$. By Claims \ref{claim:xi degree} and \ref{claim:x3,x4 degrees} we have $d(x_1) + d(x_3) \leq 2n - d(x_2) \leq \frac{4n}{3}$ and $d(x_4) \leq \frac{2n}{3}$, so $d(x_1) + d(x_3) + d(x_4) \leq 2n$. Hence, $e(G') = e(G) - d(x_1) - d(x_3) - d(x_4) + 3 \geq t_3(n) + 1 - 2n + 3 = t_3(n-3) + 1$. By the induction hypothesis, $G'$ has a chordal subgraph $H'$ with $e(H') \geq g_3(n-3) - 6$ such that $H'$ contains the edges of the clique $x_2,y_0,v,w$. Now take $H = H' + \{x_3y_0,x_4y_0\} + \{x_ix_j : 1 \leq i < j \leq 4\}$, see Figure 3(c). Then $e(H) = e(H') + 8 \geq g_3(n-3) + 8 - 6 = g_3(n) - 6$. Also, $H$ is chordal: adding the vertices $x_3,x_4,x_1$ in this order, we always add a simplicial vertex.  
	This completes the proof of the theorem.  
	\end{proof} 

\begin{figure}
\begin{subfigure}{0.3\textwidth}
		\begin{tikzpicture}[scale = 1.5]
			\coordinate (x1) at (1,1);
		\coordinate (x2) at (0,1);
		\coordinate (x3) at (0,0);
		\coordinate (x4) at (1,0);
		\coordinate (y0) at (-0.7,-0.7);
		\coordinate (z0) at (1.7,-0.7);
		\coordinate (v) at (-0.35-0.35,0.14+0.35*0.41);
		\foreach \i in {1,2,3,4}
		{
			\draw (x\i) node[fill=black,circle,minimum size=2pt,inner sep=0pt] {};
		}
		\draw (y0) node[fill=black,circle,minimum size=2pt,inner sep=0pt] {};
		\draw (z0) node[fill=black,circle,minimum size=2pt,inner sep=0pt] {};
		\draw (v) node[fill=black,circle,minimum size=2pt,inner sep=0pt] {};
		
		%	\draw (x1) -- (x2) -- (x3) -- (x4) -- (x1); \draw (x1) -- (x3); \draw (x2) -- (x4);
		%	\draw (y0) -- (x2); \draw (y0) -- (x3); \draw (y0) -- (x4);
		%	\draw (z0) -- (x1); \draw (z0) -- (x3); \draw (z0) -- (x4);
		%	\draw (v) -- (x2); \draw (v) -- (y0);
		
		\draw (x1) -- (x3) -- (x4) -- (x1); 
% 		\draw[dashed] (x2) -- (x3); \draw[dashed] (x2) -- (x4); \draw[dashed] (x2) -- (x1);
% 		\draw[dashed] (y0) -- (x2); \draw[dashed] (y0) -- (x3); \draw[dashed] (y0) -- (x4);
% 		\draw (z0) -- (x1); \draw (z0) -- (x3); \draw (z0) -- (x4);
% 		\draw[dashed] (v) -- (x2); \draw[dashed] (v) -- (y0);
        \draw[color=red] (x2) -- (x3); \draw[color=red] (x2) -- (x4); \draw[color=red] (x2) -- (x1);
		\draw[color=red] (y0) -- (x2); \draw[color=red] (y0) -- (x3); \draw[color=red] (y0) -- (x4);
		\draw (z0) -- (x1); \draw (z0) -- (x3); \draw (z0) -- (x4);
		\draw[color=red] (v) -- (x2); \draw[color=red] (v) -- (y0);
		
		\draw (x1) node[above] {$x_1$};
		\draw (x2) node[above] {$x_2$};
		\draw (x3) node[below] {$x_3$};
		\draw (x4) node[below] {$x_4$};
		\draw (y0) node[below] {$y_0$};
		\draw (z0) node[below] {$z_0$};
		\draw (v) node[left] {$v$};
		\end{tikzpicture}
		\caption{\centering Deleting $x_2,y_0,v$}
\end{subfigure}
\begin{subfigure}{0.3\textwidth}
		\hspace{0.45cm}
		%\caption{Deleting $x_2,y_0,v$ and applying induction. The added edges are dashed}
	\begin{tikzpicture}[scale = 1.5]
		\coordinate (x1) at (1,1);
		\coordinate (x2) at (0,1);
		\coordinate (x3) at (0,0);
		\coordinate (x4) at (1,0);
		\coordinate (y0) at (-0.7,-0.7);
		\coordinate (z0) at (1.7,-0.7);
		\coordinate (v) at (-0.35-0.35,0.14+0.35*0.41);
		\foreach \i in {1,2,3,4}
		{
			\draw (x\i) node[fill=black,circle,minimum size=2pt,inner sep=0pt] {};
		}
		\draw (y0) node[fill=black,circle,minimum size=2pt,inner sep=0pt] {};
		%	\draw (z0) node[fill=black,circle,minimum size=2pt,inner sep=0pt] {};
		\draw (v) node[fill=black,circle,minimum size=2pt,inner sep=0pt] {};
		
		%	\draw (x1) -- (x2) -- (x3) -- (x4) -- (x1); \draw (x1) -- (x3); \draw (x2) -- (x4);
		%	\draw (y0) -- (x2); \draw (y0) -- (x3); \draw (y0) -- (x4);
		%	\draw (z0) -- (x1); \draw (z0) -- (x3); \draw (z0) -- (x4);
		%	\draw (v) -- (x2); \draw (v) -- (y0);
		
		\draw (x2) -- (x3);
		\draw (y0) -- (x2); \draw (y0) -- (x3); 
		%	\draw (z0) -- (x1); \draw (z0) -- (x3); \draw (z0) -- (x4);
		\draw (v) -- (x2); \draw (v) -- (y0); \draw (v) -- (x3);
% 		\draw[dashed] (x1) -- (x3) -- (x4) -- (x1); \draw[dashed] (x1) -- (x2); \draw[dashed] (x4) -- (x2);
% 		\draw[dashed] (x4) -- (y0);
        \draw[color=red] (x1) -- (x3) -- (x4) -- (x1); \draw[color=red] (x1) -- (x2); \draw[color=red] (x4) -- (x2);
		\draw[color=red] (x4) -- (y0);
		
		\draw (x1) node[above] {$x_1$};
		\draw (x2) node[above] {$x_2$};
		\draw (x3) node[below] {$x_3$};
		\draw (x4) node[below] {$x_4$};
		\draw (y0) node[below] {$y_0$};
		%	\draw (z0) node[below] {$z_0$};
		\draw (v) node[left] {$v$};
		\end{tikzpicture}
		\caption{\centering Deleting $x_1,x_4$}
\end{subfigure}
\begin{subfigure}{0.3\textwidth}
		% \hspace{0.3cm}
%		\centering 
		\begin{tikzpicture}[scale = 1.5]
		\coordinate (x1) at (1,1);
		\coordinate (x2) at (0,1);
		\coordinate (x3) at (0,0);
		\coordinate (x4) at (1,0);
		\coordinate (y0) at (-0.7,-0.7);
		\coordinate (z0) at (1.7,-0.7);
		\coordinate (v) at (-0.35-0.35,0.14+0.35*0.41);
		\coordinate (w) at (-0.35-0.35*3,0.14+0.35*3*0.41);
		\foreach \i in {1,2,3,4}
		{
			\draw (x\i) node[fill=black,circle,minimum size=2pt,inner sep=0pt] {};
		}
		\draw (y0) node[fill=black,circle,minimum size=2pt,inner sep=0pt] {};
		%	\draw (z0) node[fill=black,circle,minimum size=2pt,inner sep=0pt] {};
		\draw (v) node[fill=black,circle,minimum size=2pt,inner sep=0pt] {};
		\draw (w) node[fill=black,circle,minimum size=2pt,inner sep=0pt] {};
		
		%	\draw (x1) -- (x2) -- (x3) -- (x4) -- (x1); \draw (x1) -- (x3); \draw (x2) -- (x4);
		%	\draw (y0) -- (x2); \draw (y0) -- (x3); \draw (y0) -- (x4);
		%	\draw (z0) -- (x1); \draw (z0) -- (x3); \draw (z0) -- (x4);
		%	\draw (v) -- (x2); \draw (v) -- (y0);
		
		\draw (y0) -- (x2);  
		\draw (v) -- (x2); \draw (v) -- (y0); 
		\draw (w) -- (x2); \draw (w) -- (y0); \draw (w) -- (v);
% 		\draw[dashed] (x1) -- (x3) -- (x4) -- (x1); \draw[dashed] (x1) -- (x2); \draw[dashed] (x4) -- (x2);
% 		\draw[dashed] (x3) -- (y0);
% 		\draw[dashed] (x4) -- (y0);
% 		\draw[dashed] (x3) -- (x2);
% 		\draw[dashed] (y0) -- (x3);
        \draw[color=red] (x1) -- (x3) -- (x4) -- (x1); \draw[color=red] (x1) -- (x2); \draw[color=red] (x4) -- (x2);
		\draw[color=red] (x3) -- (y0);
		\draw[color=red] (x4) -- (y0);
		\draw[color=red] (x3) -- (x2);
		\draw[color=red] (y0) -- (x3);
		
		\draw (x1) node[above] {$x_1$};
		\draw (x2) node[above] {$x_2$};
		\draw (x3) node[below] {$x_3$};
		\draw (x4) node[below] {$x_4$};
		\draw (y0) node[below] {$y_0$};
		%	\draw (z0) node[below] {$z_0$};
		\draw (v) node[below left] {$v$};
		\draw (w) node[left] {$w$};
		\end{tikzpicture}
		\caption{\centering Deleting $x_1,x_3,x_4$}
\end{subfigure}	
\caption{Applying induction after deleting certain vertices. The added edges are red.}
\end{figure}
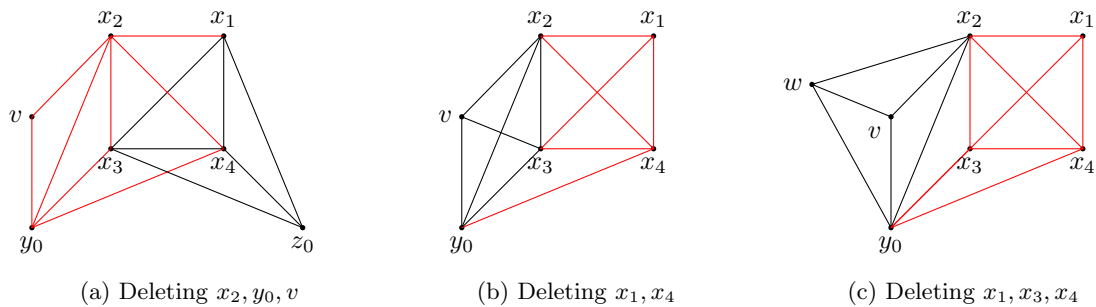

\section{Concluding remarks}
In this paper we study the maximal value $f(n,m)$ such that every graph with $n$ vertices and $m$ edges has a chordal subgraph with at least $f(n,m)$ edges.
We determine this function asymptotically for all $m$ and exactly for $m \leq t_3(n)+1$.
Our results suggest the following conjecture on the value of $f(n,m)$ for general $m$.
\begin{conjecture}\label{conj:general}
	Let $k \geq 2$, $n \geq 1$ and $t_{k}(n)+1 \leq m \leq t_{k+1}(n)$. Then 
	$$f(n,m) = \min_{t,r}(kn - t+r\big) - \binom{k+1}{2},$$ where the minimum is taken over all $t,r \geq 0$ satisfying $t_{k-1}(n-t) + t(n-t) + t_2(r) \geq m$. 
\end{conjecture}
\noindent
It seems very likely that some progress on this conjecture can be achieved using our techniques together with a careful case analysis, but it would be interesting to find a proof which avoids such a case analysis as much as possible.
	
% \bibliographystyle{abbrv}
% \bibliography{library}

\paragraph{Note added:} Zachary Hunter (personal communication) very recently improved the error term in Theorem~\ref{thm:general} from $O(\sqrt{n})$ to $O(1)$ in the special case $m = t_k(n)+1$, proving that $f(n,t_k(n)+1) = (k-1/k)n - O_k(1)$.

\paragraph{Acknowledgments:} The authors would like to thank the anonymous referee for a careful reading of the paper and useful suggestions.

\end{document}